\documentclass{amsart}

\newtheorem{thm}{Theorem}[section]
\newtheorem{lem}[thm]{Lemma}
\newtheorem{cor}[thm]{Corollary}
\newtheorem{prop}[thm]{Proposition}

\numberwithin{equation}{section}

\newcommand{\thmlabel}[1]{\label{thm:#1}} 
\newcommand{\lemlabel}[1]{\label{lem:#1}} 
\newcommand{\corlabel}[1]{\label{cor:#1}} 
\newcommand{\prplabel}[1]{\label{prp:#1}} 
\newcommand{\seclabel}[1]{\label{sec:#1}} 
\newcommand{\eqnlabel}[1]{\label{eqn:#1}} 

\newcommand{\thmref}[1]{\ref{thm:#1}}			
\newcommand{\lemref}[1]{\ref{lem:#1}}			
\newcommand{\corref}[1]{\ref{cor:#1}}			
\newcommand{\prpref}[1]{\ref{prp:#1}}			
\newcommand{\secref}[1]{\ref{sec:#1}}			
\newcommand{\eqnref}[1]{\eqref{eqn:#1}}		

\newcommand{\id}{\mathrm{id}}							
\newcommand{\ldiv}{\backslash}						
\newcommand{\rdiv}{/}											
\newcommand{\inv}{^{-1}}									
\newcommand{\sbl}[2]{\langle #1 \, | \, #2 \rangle} 
\newcommand{\setof}[2]{\{ #1 \,|\, #2 \}}	

\newcommand{\Aut}{\operatorname{Aut}}			
\newcommand{\Atp}{\operatorname{Atp}}			
\newcommand{\Mlt}{\operatorname{Mlt}}			
\newcommand{\LMlt}{\mathcal{L}}						
\newcommand{\RMlt}{\mathcal{R}}						
\newcommand{\Inn}{\operatorname{Inn}}			

\newcommand{\buchid}{\textsc{B}}					
\newcommand{\buchimp}{\buchid'}						
\newcommand{\buchimplong}{\buchid''}				
\newcommand{\buchbigid}{\textsc{\^B}}			
\newcommand{\buchbigimp}{\buchbigid'}			
\newcommand{\Buch}{\mathbf{B}}						
\newcommand{\BBuch}{\tilde{\mathbf{B}}}		
\newcommand{\Ex}{\mathbf{E}}							
\newcommand{\Mouf}{\mathbf{M}}						
\newcommand{\WIP}[1]{W$^{#1}$IP}					
\newcommand{\WWIP}{\textsc{WWIP}}					
\newcommand{\Eaut}{\textsc{E}}						
\newcommand{\dM}{\textsc{M}}							
\newcommand{\LCC}{\textsc{LCC}}						
\newcommand{\RCC}{\textsc{RCC}}						
\newcommand{\LIP}{\textsc{LIP}}						
\newcommand{\RIP}{\textsc{RIP}}						
\newcommand{\Flex}{\textsc{Flex}}					
\newcommand{\LAlt}{\textsc{LAlt}}					
\newcommand{\RAlt}{\textsc{RAlt}}					
\newcommand{\Extra}{\textsc{Ex}}					

\newcommand{\vhi}{\varphi}								
\newcommand{\eps}{\varepsilon}						

\newcommand{\textand}{\quad\text{and}\quad}
\newcommand{\frall}{\quad\text{for all}\quad}

\begin{document}

\title[Buchsteiner loops]
{Buchsteiner loops}
\author{Piroska Cs\"{o}rg\H{o}}
\address{Department of Mathematics \\
E\"{o}tv\"{o}s University \\
P\'{a}zm\'{a}ny P\'{e}ter s\'{e}t\'{a}ny 1/C \\
H-1117 Budapest, Hungary}
\email{ska@cs.elte.hu}
\author {Ale\v{s} Dr\'{a}pal}
\address{Department of Mathematics \\
Charles University \\
Sokolovsk\'{a} 83 \\
186 75 Praha 8, Czech Rep.}
\email{drapal@karlin.mff.cuni.cz}
\author{Michael K. Kinyon}
\address{Department of Mathematics \\
University of Denver \\
2360 S. Gaylord St. \\
Denver, Colorado 80208, U.S.A.}
\email{mkinyon@math.du.edu}

\thanks{The first author (P. Cs\"org\H o) was supported by Hungarian National
Foundation for Scientific Research, grants T049841 and T038059. The second
author (A. Dr\'apal) was supported by institutional grant MSM 0021620839.
The third author (M. Kinyon) was supported by the Eduard \v{C}ech Center
LC505. This paper was written
while the second author was a Fulbright Research Scholar at the University of
Wisconsin-Madison.}

\begin{abstract}
Buchsteiner loops are those which satisfy the identity
$x\ldiv (xy \cdot z) = (y \cdot zx)\rdiv x$. We show that a
Buchsteiner loop modulo its nucleus is an abelian group
of exponent four, and construct an example where the factor
achieves this exponent.
\end{abstract}

\maketitle

A \emph{loop} $(Q,\cdot)$ is a set $Q$ together with a binary operation
$\cdot$ such that for each $a,b\in Q$, the equations $a\cdot x = b$
and $y\cdot a = b$ have unique solutions $x,y\in Q$, and such that
there is a neutral element $1\in Q$ satisfying $1\cdot x = x\cdot 1 = x$
for every $x\in Q$. Standard references in loop theory are
\cite{bel,bruck,pflug}.

The variety (\emph{i.e.}, equational class) of all loops being too broad for
a detailed structure theory, most investigations focus on particular classes
of loops.
In this paper we investigate a variety of loops which has not hitherto received
much attention, despite the fact that it is remarkably rich in structure, namely
the variety defined by the identity
\begin{equation*}
x\ldiv (xy \cdot z) = (y \cdot zx)\rdiv x .
\tag{\buchid}
\end{equation*}
Here $a\ldiv b$ denotes the unique solution $x$ to $a\cdot x = b$, while
$b\rdiv a$ denotes the unique solution $y$ to $y\cdot a = b$.
We call (\buchid) the \emph{Buchsteiner law} and a loop satisfying it a
\emph{Buchsteiner loop} since Hans-Hennig
Buchsteiner seems to have been the first to notice their importance \cite{buch}.

The Buchsteiner law (\buchid) is easily seen to be equivalent
to each of the following:
\begin{align*}
xy\cdot z &= xu \quad &\Longleftrightarrow \qquad y\cdot zx &= ux
\quad & \text{ for all }x,y,z,u\,, \text{and}
\tag{\buchimp}\\
xy\cdot z &= xu\cdot v \quad &\Longleftrightarrow \qquad y\cdot zx &= u\cdot vx
\quad & \text{ for all }x,y,z,u,v\,.
\tag{\buchimplong}
\end{align*}
Both the identity (\buchid) and the implications (\buchimp), (\buchimplong) will prove
useful in what follows.

Buchsteiner loops can be understood in terms of coinciding
left and right principal isotopes or in terms of autotopisms
that identify the left and the right nucleus. Among the
equalities that can be obtained by nuclear identification
this is the only one which has not been subjected to a systematic
study. (The other equalities are the left and right Bol laws,
the Moufang laws, the extra laws, and the laws of left and
right conjugacy closedness \cite{dni}.)

All Buchsteiner loops are $G$-loops (a loop $Q$ is said to be a $G$-loop if
every loop isotope of $Q$ is isomorphic to $Q$). Groups are $G$-loops, and
it is well-known and easy to see that conjugacy closed loops are $G$-loops
as well. There are also various subclasses of Moufang loops which turn out
to be $G$-loops, such as the class of all simple Moufang loops, or the
variety of all $M_k$-loops for $k \not\equiv 1$ (mod $3$) \cite{cheinpflug}. Buchsteiner loops seem to be the only other known class
of $G$-loops with a concise equational definition. Such classes are of
considerable interest since it is known that $G$-loops cannot be described
by  first-order sentences \cite{kun}.

Our main results are the statement that in every Buchsteiner loop $Q$, the
nucleus $N$ is a normal subloop such that the factor loop $Q/N$ is an
abelian group of exponent $4$, and the construction of an example in which
that exponent is achieved.

Our investigations have been helped immensely by recent progress on
conjugacy closed loops \cite{kun,ccm,ccd,lc2,ccp,lce,cpq}. Buchsteiner did not
work with the notion of conjugacy closedness, and it turns out that the
examples he constructed in \cite{buch} are all conjugacy closed. This paper
thus seems to present the first example of a proper (non-CC) Buchsteiner
loop. A deeper understanding of this will require further study. We are
inclined to believe that Buchsteiner's speculations in \cite{buch} concerning
the connections to nonassociative division algebras (and, indirectly, to
projective planes) may prove to have been prescient.

To obtain the present results we had to use the concept of
\emph{doubly weak inverse property} (WWIP, for short), which naturally
generalizes the classical weak inverse property (WIP) of Osborn \cite{osb}.
WWIP loops can be also seen as a special case of the more general notion of
$m$-inverse loops \cite{kk,ks}. We believe that this is the first instance
where this concept found a highly nontrivial natural application in an
equational theory. The main result of \S\secref{WWIP} is the proof that each
Buchsteiner loop is a WWIP loop.

In \S\secref{isotopes}, we apply properties of WWIP loops to show that
each Buchsteiner loop $Q$ is a G-loop and that $Q/N$ is an abelian group.
In \S\secref{WIP}, we consider the special case of Buchsteiner loops
with the WIP, and show that these are exactly WIP CC loops.

To get the aforementioned restriction
on the exponent of $Q/N$ we use associator calculus, which is
developed in \S\secref{associator}. We suspect that future results on
Buchsteiner loops will require further and finer calculations with associators.
In \S\secref{compatible} and \S\secref{extend} we construct an example of a Buchsteiner loop $Q$ on 1024 elements such that $Q/N$ is indeed of the minimal exponent $4$. Further, $Q$ has a factor of order $64$ with this same
property.

We will discuss the relationship of our work to that of Buchsteiner \cite{buch}
and Basarab \cite{bao} in \S\secref{history}, and we conclude by announcing
some further results and stating several problems in \S\secref{conclusions}.

We have tried to write this paper in a way that is accessible to
researchers who are not specialists in loop theory.
We define all notions we need and when mentioning basic
properties which are easy to show we usually offer a proof.

We are pleased to acknowledge the assistance of McCune's automated
theorem proving program Prover9 \cite{Prover9}, which was particularly
invaluable for the results in \S\secref{WWIP}.

\section{Preliminaries}
\seclabel{basics}

In this section we introduce many of the basic tools of loop theory,
and examine how they are used in Buchsteiner loops.

\subsection{Multiplication groups}

Let $Q$ be a loop. For each $x\in Q$, the \emph{left} and \emph{right translation}
maps $L_x, R_x : Q\to Q$ are defined by $L_x y = xy$ and $R_x y = yx$,
respectively. For any $S\subseteq Q$, set
\[
L_{(S)}  = \setof{L_s}{s \in S}
\qquad \text{and} \qquad
R_{(S)}  = \setof{R_s}{s\in S} \;.
\]
The group generated by both types of translations
\[
\Mlt Q = \langle L_{(Q)}, R_{(Q)}\rangle = \sbl{L_x, R_x}{x\in Q}
\]
is called the \emph{multiplication group}, while the \emph{left} and
\emph{right multiplication groups} are defined, respectively, by
\[
\LMlt = \LMlt(Q) = \langle L_{(Q)}\rangle = \sbl{L_x}{x \in Q}
\textand
\RMlt = \RMlt(Q) = \langle R_{(Q)}\rangle = \sbl{R_x}{x \in Q}\,.
\]
Working with left and right translation maps allows many computations in loop
theory to be carried out in groups. For instance, rewriting the Buchsteiner
law (\buchid) in terms of translations immediately yields

\begin{lem}
\lemlabel{buch-trans}
In a loop $Q$, the Buchsteiner law \emph{(\buchid)} is equivalent
to each of the following:
\begin{align}
L_x\inv R_z L_x &= R_x\inv R_{zx} \text{ for all } x,z\in Q\,,
\eqnlabel{buch-trans1} \\
R_x\inv L_y R_x &= L_x\inv L_{xy} \text{ for all } x,y\in Q\,.
\eqnlabel{buch-trans2}
\end{align}
\end{lem}

The following alternative form of Lemma \lemref{buch-trans} is also useful.

\begin{prop}
\prplabel{trans-char}
In a loop $Q$, the Buchsteiner law \emph{(\buchid)} is equivalent
to each of the following:
\begin{align}
R_x R_{(Q)}^{L_x} = R_{(Q)} \text{ for all } x\in Q\,, 
\eqnlabel{trans-var1} \\
L_x L_{(Q)}^{R_x} = L_{(Q)} \text{ for all } x\in Q\,.
\eqnlabel{trans-var2}
\end{align}
\end{prop}

\begin{proof}
If \eqnref{buch-trans1} holds, then so does \eqnref{trans-var1}.
Conversely, if \eqnref{trans-var1} holds, then
for each $y\in Q$, there exists $z\in Q$ such that
$L_x\inv R_y L_x = R_x\inv R_z$. Applying both sides to $1\in Q$,
we get $yx = z$, and so $L_x\inv R_y L_x = R_x\inv R_{yx}$, which is
\eqnref{buch-trans1}. The proof of the equivalence of
\eqnref{buch-trans2} and \eqnref{trans-var2} is similar.
\end{proof}

As subgroups of $\Mlt Q$, neither $\LMlt$ nor $\RMlt$ has to be normal
in the general case. However, both Lemma \lemref{buch-trans} and
Proposition \prpref{trans-char} have the following immediate consequence.

\begin{cor}
\corlabel{normal}
In a Buchsteiner loop $Q$,
$\LMlt$ and $\RMlt$ are normal subgroups of $\Mlt Q$.
\end{cor}

\subsection{Inner mappings}

The stabilizer in $\Mlt Q$ of the neutral element of a loop $Q$ is called
the \emph{inner mapping group} of $Q$:
\[
\Inn Q = (\Mlt Q)_1 = \{ \varphi\in \Mlt Q \, |\, \varphi(1) = 1\} ,
\]
while the \emph{left} and \emph{right inner mapping groups} are defined,
respectively, by
\[
\LMlt_1 = \LMlt \cap \Inn Q
\qquad \text{and}\qquad
\RMlt_1 = \RMlt \cap \Inn Q\,.
\]
Generators for these groups are defined as follows:
\[
L(x,y) = L_{xy}\inv L_x L_y\,,
\qquad
T_x = R_x\inv L_x\,,
\qquad
R(x,y) = R_{yx}\inv R_x R_y
\]
for $x,y\in Q$. Then it turns out \cite{bel, bruck} that
\begin{gather*}
\Inn Q = \sbl{T_x, L(x,y), R(x,y)}{x,y\in Q}\,, \\
\LMlt_1 = \sbl{L(x,y)}{x,y\in Q}\,, \qquad
\RMlt_1 = \sbl{R(x,y)}{x,y\in Q}\,.
\end{gather*}

In Buchsteiner loops, matters are somewhat simpler. By rewriting
\eqnref{buch-trans1} and \eqnref{buch-trans2} in the form
\begin{equation}
\eqnlabel{e14}
[R_z,L_x] = R_z\inv R\inv_x R_{zx} \textand [L_y,R_x] = L_y\inv L_x\inv L_{xy}\,,
\end{equation}
we get

\begin{thm}
\thmlabel{26}
Let $Q$ be a Buchsteiner loop. Then for all $x,y \in Q$,
\[
R(x,y) = [L_x,R_y] = L(y,x)\inv \; .
\]
In particular, $\LMlt_1 = \RMlt_1$.
\end{thm}

\subsection{Nuclei, the center, and related subloops}

The \emph{left}, \emph{middle}, and \emph{right} nucleus
of a loop $Q$ are defined, respectively, by
\begin{align*}
N_{\lambda} &= N_{\lambda}(Q) =
\{ a \in Q \,|\,  a\cdot xy = ax\cdot y \quad \forall x,y\in Q\}\,, \\
N_{\mu} &= N_{\mu}(Q) =
\{ a \in Q \,|\,  x\cdot ay = xa\cdot y \quad \forall x,y\in Q\}\,, \\
N_{\rho} &= N_{\rho}(Q) =
\{ a \in Q \,|\,  x\cdot ya = xy\cdot a \quad \forall x,y\in Q\},,
\end{align*}
while the intersection
\[
N = N(Q) = N_\lambda \cap N_\mu \cap N_\rho
\]
is called the \emph{nucleus} of $Q$. We observe that
$N_{\lambda}$ is the set of fixed points of $\RMlt_1$,
$N_{\rho}$ is the set of fixed points of $\LMlt_1$, and
$N_{\mu}$ is the set of fixed points of the group
$\sbl{ \lbrack R_x, L_y \rbrack}{x,y\in Q}$.

It is easy to prove that each of the nuclei is a subgroup
(associative subloop) of $Q$.
However, in general, the nuclei need not coincide.

\begin{lem}
\lemlabel{cyclic-assoc}
Let $Q$ be a Buchsteiner loop. For all $x,y,z\in Q$,
\[
x\cdot yz = xy\cdot z
\qquad \Longleftrightarrow \qquad
y\cdot zx = yz\cdot x
\qquad \Longleftrightarrow \qquad
z\cdot xy = zx\cdot y\,.
\]
\end{lem}

\begin{proof}
This is an immediate consequence of (\buchimp).
\end{proof}

\begin{cor}
\corlabel{equal-nuc}
In a Buchsteiner loop, $N_{\lambda} = N_{\mu} = N_{\rho}$.
\end{cor}

This corollary also follows from Theorem \thmref{26} and the
characterizations of the nuclei as fixed point sets of subgroups
of $\Inn Q$.

A subloop $S$ of a loop $Q$ is \emph{normal} if it is
invariant under the action of $\Inn Q$, or equivalently,
if it is a block of $\Mlt Q$ containing the neutral
element. Then one can define a quotient structure $Q/S$
over the associated block system.
In general, the nuclei of a loop are not necessarily normal
subloops.

Let $Q$ be a loop, and suppose $\phi\in \Mlt Q$ satisfies
$\phi R_x = R_x \phi$ for all $x\in Q$. Applying both sides
to $1$ gives $\phi(x) = ax$ where $a = \phi(1)$, and so
$\phi = L_a$. However, $L_aR_x = R_xL_a$
for all $x\in Q$ exactly when $a \in N_\lambda$. Similarly,
$\psi\in \Mlt Q$ satisfies $\psi L_x = L_x \psi$ for all
$x\in Q$ if and only if $\psi = R_b$ for some
$b\in N_{\rho}$.

Then we can express the observations of the preceding paragraph
as part (i) of the following.

\begin{lem}
\lemlabel{centralizers}
Let $Q$ be a loop. Then
\begin{enumerate}
\item\quad $C_{\Mlt Q}(\RMlt) = L_{(N_{\lambda})}$
and $C_{\Mlt Q}(\LMlt) = R_{(N_{\rho})}$.
\item\quad If $\RMlt \unlhd \Mlt Q$, then $L_{(N_{\lambda})} \unlhd \Mlt Q$
and $N_{\lambda} \unlhd Q$.
\item\quad If $\LMlt \unlhd \Mlt Q$, then $R_{(N_{\rho})} \unlhd \Mlt Q$
and $N_{\rho} \unlhd Q$.
\end{enumerate}
\end{lem}

\begin{proof}
Parts (ii) and (iii) follow because the centralizer of a normal
subgroup is again a normal subgroup, and orbits of normal subgroups
of permutation groups form block systems.
\end{proof}

\begin{cor}
\corlabel{15}
Let $Q$ be a Buchsteiner loop. Then $N$ is a normal subloop of $Q$,
and $L_{(N)}$ and $R_{(N)}$ are normal subgroups of $\Mlt Q$.
\end{cor}

\begin{proof}
This follows from Corollaries \corref{normal} and \corref{equal-nuc},
and Lemma \lemref{centralizers}.
\end{proof}

In an arbitrary loop $Q$, the set
\[
C(Q) = \{ a \in Q \,|\, ax = xa, \; \forall x\in Q\}
= \{ a\in Q \,|\, T_a = \id_Q \}
\]
is not necessarily a subloop. The \emph{center} of
$Q$ is defined as
\[
Z(Q) = N(Q)\cap C(Q)\,.
\]
The center is exactly the set of fixed points of the inner
mapping group $\Inn Q$, and thus is a normal, abelian subgroup
of $Q$.

\begin{lem}
\lemlabel{center}
In a Buchsteiner loop $Q$, $Z(Q) = C(Q)$.
\end{lem}

\begin{proof}
Fix $c\in C(Q)$, and in \eqnref{buch-trans1}, take $x = c$.
Since $L_c = R_c$, we have
$R_c\inv R_z R_c = R_c\inv R_{zc}$, and so
$R_z R_c = R_{cz}$ for all $z\in Q$. This says that
$c\in N_{\mu} = N$ (Corollary \corref{equal-nuc}),
and so $c\in Z(Q)$.
\end{proof}

\subsection{Isotopisms}

An \emph{isotopism} of loops $(Q,\ast)$ and $(Q,\cdot)$ with the same
underlying set $Q$ is a triple $(\alpha,\beta,\gamma)$ of permutations
of $Q$ satisfying
\begin{equation}
\eqnlabel{e3}
\alpha(x)\cdot \beta(y) = \gamma(x\ast y)
\end{equation}
for all $x,y\in Q$. In this case,  $(Q,\ast)$ and $(Q,\cdot)$
are said to be \emph{isotopic}. Two types of isotopisms are of primary
interest in loop theory, namely, principal isotopisms and autotopisms.
We begin with the former.

An isotopism $(\alpha,\beta,\gamma)$ is called \emph{principal} if
$\gamma = \id_Q$. In such a case, if $1\in Q$ is the neutral
element of $(Q,\ast)$, and if we set $a = \alpha(1)$ and $b = \beta(1)$,
then \eqnref{e3} becomes
\[
x \ast y = (x \rdiv b) \cdot (a \ldiv y)
\]
for all $x,y\in Q$. Here $\ldiv$ and $\rdiv$ are the left and right
division operations in $(Q,\cdot)$. The loop $(Q,\ast)$ is then
called a \emph{principal isotope} of $(Q,\cdot)$.

Let $(Q,\cdot)$ be a loop and fix $e\in Q$. Define on $Q$ loop operations
$\ast$ and $\circ$ by
\begin{equation}
\eqnlabel{e2}
x\ast y = (x\rdiv e) \cdot y
\qquad \text{and} \qquad
x\circ y = (x \cdot ye)\rdiv e\,.
\end{equation}
Thus $(Q,\ast)$ is a principal isotope of $Q$ with neutral element $e$,
while $1$ is the neutral element of $(Q,\circ)$.

\begin{lem}
\lemlabel{011}
The right translation $R_e: x\mapsto xe$ yields an isomorphism
$(Q,\circ) \cong (Q,\ast)$.
\end{lem}

\begin{proof}
Indeed, $R_e(x\circ y) = x \cdot ye = ((xe)\rdiv e) \cdot ye =
R_e(x) \ast R_e(y)$.
\end{proof}

Lemma \lemref{011} clearly has a mirror version, which associates
the operations $x \cdot (e\ldiv y)$ and $e\ldiv (ex \cdot y)$.
In the loop theory literature, the operations $e\ldiv (ex\cdot y)$ and
$(x\cdot ye)\rdiv e$ are sometimes called the ``left derivative'' and
``right derivative'' at $e$, respectively \cite{pflug}. However,
invoking Lemma \lemref{011}, we will call the operation
$(x \cdot ye)\rdiv e$ the \emph{right isotope at $e$} and the operation
$e\ldiv (ex\cdot y)$ the \emph{left isotope at $e$}.  We obviously have

\begin{lem}
\lemlabel{12}
A loop $Q$ satisfies the Buchsteiner law {\rm{(}}\buchid{\rm{)}} if and only
if for every $e \in Q$, the left and right isotopes at $e$ coincide.
\end{lem}

When dealing with a Buchsteiner loop, we can thus drop the left/right
distinction and refer simply an \emph{isotope at} $e$. We will denote
this isotope by $Q[e]$.

\subsection{Autotopisms}

An isotopism $(\alpha,\beta,\gamma)$ of a loop $(Q,\cdot)$ to
itself is called an \emph{autotopism}. The set $\Atp Q$ of all
autotopisms of a loop $Q$ is a group, and a permutation $\alpha$
of $Q$ is an automorphism, that is, $\alpha \in \Aut Q$,
if and only if $(\alpha,\alpha,\alpha)\in \Atp Q$.

\begin{lem}
\lemlabel{13}
A loop $Q$ satisfies the Buchsteiner law {\rm{(}}\buchid{\rm{)}} if and only if
\[
\Buch(x) = ( L_x, R_x\inv, L_x R_x\inv )
\]
is an autotopism for all $x\in Q$.
\end{lem}

\begin{proof}
The triple $\Buch(x)$ yields an autotopism if and only if
$(xy)(z\rdiv x) = x((yz)\rdiv x)$ for all $x,y,z \in Q$. By replacing
$z$ with $zx$ we get $xy\cdot z = x((y\cdot zx)\rdiv x)$, which is
equivalent to (\buchid).
\end{proof}

The nuclei can be characterized in terms of autotopisms.

\begin{lem}
\lemlabel{nuc-aut}
Let $Q$ be a loop. For $a\in Q$,
\begin{align*}
a\in N_{\lambda} &\Longleftrightarrow (L_a,\id_Q,L_a)\in \Atp Q\,, \\
a\in N_{\mu} &\Longleftrightarrow (R_a,L_a\inv,\id_Q)\in \Atp Q\,, \\
a\in N_{\rho} &\Longleftrightarrow (\id_Q,R_a,R_a)\in \Atp Q\,.
\end{align*}
\end{lem}

\begin{proof}
For instance, $x\cdot ay = xa \cdot y$ for all $x,y \in Q$
if and only if $xy = (xa)\cdot(a\ldiv y)$ for all $x,y \in Q$, which
gives the characterization of $N_{\mu}$. The other cases are similar.
\end{proof}

Let us denote the triples of permutations occurring in the preceding
lemma as follows:
\[
\alpha_{\lambda}(x) = (L_x, \id, L_x)\,,
\qquad
\alpha_{\mu}(x) = (R_x, L_x\inv, \id)\,,
\qquad
\alpha_{\rho}(x) = (\id, R_x, R_x)\,.
\]
Suppose that $Q$ is a loop in which there exist $\eps,\eta \in \{-1,1\}$
and $\xi,\chi \in \{\lambda,\mu,\rho\}$ such that
$\alpha^\eps_\xi(a)\, \alpha^\eta_\chi(a)$ is an autotopism
for all $a\in Q$. It is then clear that
$a \in N_{\xi}(Q) \Leftrightarrow a \in N_{\chi}(Q)$.

Say that an identity can be obtained by \emph{nuclear identification}
if it can be expressed by an autotopism of the form
$\alpha^\eps_\xi(a)\, \alpha^\eta_\chi(a)$. From Lemma \lemref{13},
we see that the Buchsteiner law can be obtained in this form, as
\[
\Buch(x) = (L_x, R_x\inv, L_x R_x\inv) =
(L_x, \id_Q, L_x) (\id_Q, R_x, R_x)\inv =
\alpha_{\lambda}(x) \alpha_{\rho}(x)\inv .
\]
In particular, this implies the already noted fact that
$N_{\lambda} = N_{\rho}$ (Corollary \corref{equal-nuc}).
We have already mentioned in the introduction that the most
frequently studied varieties of loops are those that can be
described by identities which follow from a nuclear identification.
This concept is studied in detail in \cite{dni}.

The following proposition will be crucial to a later description
of isotopes of Buchsteiner loops.

\begin{prop}
\prplabel{atp-cent}
Let $Q$ be a loop. For a permutation $f$ of $Q$,
\begin{enumerate}
\item $(f,\id_Q,f)\in \Atp Q \Leftrightarrow f$ centralizes $R_{(Q)}
\Leftrightarrow f = L_a$ where $a = f(1)\in N_{\lambda}$\,,
\item $(\id_Q,f,f)\in \Atp Q \Leftrightarrow f$ centralizes $L_{(Q)}
\Leftrightarrow f = R_a$ where $a = f(1)\in N_{\rho}$\,.
\end{enumerate}
\end{prop}

\begin{proof}
$(f,\id_Q,f)$ is an autotopism if and only if for each $x,y\in Q$,
$f(x)y = f(xy)$, that is, $R_y f(x) = f R_y(x)$. This shows the
first equivalence of (i). Take $x = 1$ and set $a = f(1)$ to get
$f(y) = ay$ for all $y$, that is, $f = L_a$. That $a\in N_{\lambda}$
follows from Lemma \lemref{nuc-aut}. Conversely, if $f = L_a$ for
some $a\in N_{\lambda}$, then $f$ centralizes $R_{(Q)}$ by
Lemma \lemref{centralizers}. This establishes (i), and the proof
of (ii) is similar.
\end{proof}

Let us give a classical example how the notion of an autotopism
can be used.

\begin{lem}
\lemlabel{16}
Let $Q$ be a loop, let $(\alpha,\beta,\gamma)\in \Atp Q$ satisfy
$\alpha(1) = 1$ and set $c = \beta(1)$. Then
\begin{enumerate}
\item\quad $\gamma = \beta = R_c \alpha$,
\item\quad $\alpha$ is an isomorphism from $Q$ to the right
isotope at $c$,
\item\quad $\alpha\in \Aut Q$ if and only if $c\in N_{\rho}$.
\end{enumerate}
\end{lem}

\begin{proof}
We have $\gamma(x) = \alpha(1)\beta(x) = \beta(x)$ for all $x \in Q$,
and $\alpha(x)\cdot c = \beta(x)$, establishing (i). Next,
$\alpha(x)(\alpha(y)\cdot c) = \alpha(xy)\cdot c$ for all $x,y \in Q$,
which gives (ii). If $c \in N_\rho$, then we get
$\alpha(x)\alpha(y) = \alpha(xy)$ for all $x,y \in Q$. On the other hand
the latter equality implies $c \in N_\rho$ since $\alpha$ is a permutation
of $Q$. This proves (iii).
\end{proof}

In the loop theory literature, the permutation $\alpha$ described in
the preceding lemma is sometimes called a ``right pseudoautomorphism''
with companion $c$.

At this point, and throughout the rest of the paper, it will be useful
to introduce notation for left and right inverses. In a loop $Q$,
let
\[
I(x) = x\ldiv 1 \qquad\text{and}\qquad J(x) = 1\rdiv x
\]
for all $x\in Q$. (Here we follow the notation of \cite{bel}.)
Thus $x\cdot I(x) = J(x)\cdot x = 1$,
and $J = I\inv$. In general, $I(x)$ and $J(x)$ can differ,
but when they coincide, we write $x\inv = I(x) = J(x)$.
In that case, $(x\inv)\inv = x$.

\begin{cor}
\corlabel{buch-pseudo}
Let $Q$ be a Buchsteiner loop, and fix $x,y\in Q$. Then
\begin{enumerate}
\item\quad $L(x,y)$ is an isomorphism from $Q$ to the
isotope at $(J(y)\rdiv x)\cdot xy$,
\item\quad $L(x,y)\in \Aut Q$ if and only if
$(J(y)\rdiv x)\cdot xy \in N$.
\end{enumerate}
\end{cor}

\begin{proof}
We compose Buchsteiner autotopisms to get
\[
\Buch(xy)\inv \Buch(x) \Buch(y) =
(L(x,y),\, R_{xy} R_x\inv R_y\inv,\, R_{xy} L_{xy}\inv L_x R_x\inv L_y R_y\inv ).
\]
The first component fixes $1$ and so the lemma applies.
\end{proof}

Incidentally, the equality of the
second and third components of the autotopism in the preceding proof
gives another proof of Theorem \thmref{26}.

A loop $Q$ is said be an $A_{\ell}$-\emph{loop} if $\LMlt_1 \le \Aut Q$, that is,
if $L(x,y) \in \Aut Q$ for all $x,y \in Q$. An $A_r$-\emph{loop} is
similarly defined, and we will say that a loop satisfying both properties
is an $A_{\ell,r}$-\emph{loop}. By Theorem \thmref{26}, the $A_{\ell}$ and $A_r$
properties are equivalent in Buchsteiner loops. Corollaries \corref{15} and
\corref{buch-pseudo} imply that a Buchsteiner loop is an $A_{\ell,r}$-loop
if and only if $Q/N$ satisfies $(J(y)\rdiv x)\cdot xy = 1$ for all $x,y$.
Later we will show that $Q/N$ satisfies a much stronger property, and so
every Buchsteiner loop will turn out to be an $A_{\ell,r}$-loop.

In the meantime, Corollary \corref{buch-pseudo} gives us a useful family of
automorphisms. Let
\[
\tag{\Eaut}
E_x = L(J(x),x) = L_{J(x)} L_x \,,
\]
for each $x$ in a loop $Q$.

\begin{lem}
\lemlabel{E}
Let $Q$ be a Buchsteiner loop. Then for each $x\in Q$,
\begin{enumerate}
\item\quad $E_x = R(x,J(x))\inv = [L_x , R_{J(x)}]\inv$,
\item\quad $E_x \in \Aut Q$,
\item\quad $E_{J(x)} = E_x = E_{I(x)}$,
\item\quad $E_x I^n(x) = I^{n-2}(x)$ and $E_x J^n(x) = J^{n+2}(x)$ for each integer $n$,
\item\quad $E_x = [L_x\inv, R_x\inv]$,
\item\quad $E_x = L_x R_x R(x,x) R_x\inv L_x\inv = R_x L_x R(x,x) L_x\inv R_x\inv$, and
\item\quad $E_x\inv = L_x R_x L(x,x) R_x\inv L_x\inv = R_x L_x L(x,x) L_x\inv R_x\inv$
\end{enumerate}
\end{lem}

\begin{proof}
Part (i) is a specialization of Theorem \thmref{26}
to the present setting. Part (ii) follows from
Corollary \corref{buch-pseudo}(ii). For (iii), we use
part (ii) to compute
\[
J(x)(x\cdot I(x)y) =
E_x (I(x)y) = E_x I(x)\cdot E_x (y) = J(x)\cdot E_x (y)\,.
\]
Canceling, we have $E_x = L(x,I(x)) = E_{I(x)}$, and the
other equality of (iii) follows from replacing $x$ with $J(x)$.
For (iv), we use (iii) to compute $E_x I^n(x) = E_{I^{n-2}(x)} I^n(x)
= I^{n-2}(x)\cdot I^{n-1}(x) I^n(x) = I^{n-2}(x)$, and
the other equality follows from $J = I\inv$.

Since $xy = x I(x)\cdot xy$, applying (\buchimp) gives
$yx = I(x) (xy\cdot x)$, that is, $R_x = L_{I(x)} R_x L_x$.
Multiplying on the left by $L_x$ and rearranging, we have (v).
For (vi), we use Theorem \thmref{26} and (v):
$R(x,x) = L_x\inv R_x\inv L_x R_x = R_x\inv L_x\inv E_x L_x R_x$.
The mirror of this argument yields the other equality of (vi).
Finally, (vii) is obtained from inverting (vi) and using
Theorem \thmref{26}.
\end{proof}

In \S\secref{elements} we shall need another easy general result about
autotopisms:

\begin{lem}
\lemlabel{18}
Let $Q$ be a loop, and let $\alpha$ and $\beta$ be permutations of $Q$.
Suppose that $\beta(1) = 1$. Then the triple $(\alpha,\beta,\beta)$ is
an autotopism if and only if $\alpha = \beta \in \Aut Q$.
\end{lem}

\begin{proof}
Only the direct implication needs a proof. If $(\alpha,\beta,\beta)$
is an autotopism, then $\alpha(x) = \alpha(x)\cdot 1 = \alpha(x)\beta(1)
= \beta(x\cdot 1) = \beta(x)$ for all $x\in Q$.
\end{proof}

\section{Special elements and CC loops}
\seclabel{elements}

An element $a$ of a loop $Q$ is said to have the
\emph{left inverse property} (LIP), or to be an
\emph{LIP element}, if there exists $b \in Q$ such
that $L_a\inv = L_b$, that is, $L_a L_b = L_b L_a
= \id_Q$. Applying this to $1\in Q$ gives
$ab = ba = 1$, and so $b = I(a) = J(a) = a\inv$.
The left inverse property can be thus expressed as
\[
\tag{\LIP}
L_a L_{a\inv} = L_{a\inv} L_a = L(a\inv,a) = \id_Q\,.
\]
In particular, $a$ is an LIP element if and only if $a\inv$
is an LIP element.

Similarly, $a\in Q$ is said to have the
\emph{right inverse property} (RIP),
or to be an \emph{RIP element}, if
\[
\tag{\RIP}
R_a R_{a\inv} = R_{a\inv} R_a = R(a,a\inv) = \id_Q\,.
\]
If $a\in Q$ is both an LIP and RIP element, then we will refer
to it simply as an \emph{inverse property} (IP) \emph{element}.

An element $a$ in a loop $Q$ is said to be \emph{flexible}
if $a\cdot xa = ax \cdot a$ for all $x \in Q$, that is,
if and only if
\[
\tag{\Flex}
L_a R_a = R_a L_a\,.
\]
Setting $x = J(a)$ and canceling shows that each flexible
element $a$ satisfies $I(a) = J(a)$.

An element $a$ in a loop $Q$ is said to be
\emph{left alternative} if $a\cdot ax = a^2 x$ for all $x \in Q$,
and \emph{right alternative} if $xa\cdot a = xa^2$ for all $x\in Q$.
Equivalently, these are given, respectively, by
\[
(\LAlt) \qquad L_a^2 = L_{a^2} \qquad\qquad\qquad
(\RAlt) \qquad R_a^2 = R_{a^2}\,.
\]

An element $a$ in a loop $Q$ is said to be an \emph{extra} element if
it satisfies $a(y\cdot za) = (ay \cdot z)a$ for all $y,z \in Q$, that
is, if and only if the triple of permutations
\[
\tag{\Extra}
\Ex(a) = ( L_a , \, R_a\inv , \, R_a\inv L_a )
\]
is an autotopism of $Q$. Setting $z = 1$, one
has that each extra element is flexible. Setting
$z = a\inv$ shows that each extra element is an LIP
element, and setting $y = a\inv$ gives that
each extra element is also an RIP element.

\begin{prop}
\prplabel{22}
Let $Q$ be a Buchsteiner loop. For an element $a\in Q$, each of the
following is equivalent: (i) $a$ has the LIP, (ii) $a$ has the RIP,
(iii) $a$ is flexible, (iv) $a$ is left alternative, (v) $a$ is
right alternative, (vi) $a$ is extra.
\end{prop}

\begin{proof}
Since $E_a = R_{J(a)}\inv R_a\inv = [L_a\inv , R_a\inv ]$ by
Lemma \lemref{E}, we have the equivalence of (i), (ii), and (iii).
The equivalence of (iii), (iv), and (v)
follows from Lemma \lemref{cyclic-assoc}.

We have already noted that (vi) implies (iii).
Conversely, if (iii) holds, that is, if $a$ is flexible, then
$L_a R_a^{-1} = R_a^{-1} L_a$, and so by examining the third components
of autotopisms, we have $\Buch(a) = \Ex(a)$. Thus (iii) implies (vi).
\end{proof}

An element $a$ of a loop $Q$ is said to be a \emph{Moufang} element if it
satisfies $a (xy\cdot a) = ax\cdot ya$ for all $x,y\in Q$, that is, if the
triple
\[
\Mouf(a) = ( L_a , \, R_a , \, L_a R_a )
\]
is an autotopism. It is immediate that every Moufang element is flexible.

\begin{lem}
\lemlabel{mouf}
In a Buchsteiner loop $Q$, an element $a$ is Moufang if and only if it
is extra and $a^2 \in N(Q)$.
\end{lem}

\begin{proof}
Since a Moufang element is flexible, it is extra by Proposition \prpref{22}.
Then $\Mouf(a) \Ex(a) = ( L_a^2 , \id_Q , L_a^2 )
= ( L_{a^2} , \id_Q , L_{a^2} )$, since $a$ is left alternative. By
Lemma \lemref{nuc-aut} and Corollary \corref{equal-nuc},
$a^2\in N_{\lambda} = N$. Conversely, if $a$ is extra and $a^2 \in N$,
then $( L_{a^2} , \id_Q , L_{a^2} ) \Ex(a)\inv
= ( L_a^2 L_a\inv , R_a\inv , L_a^2 L_a\inv R_a ) = \Mouf(a)$ is an
autotopism, and so $a$ is Moufang.
\end{proof}

Neither the set of extra elements nor the set of Moufang elements in
a Buchsteiner loop is necessarily a subloop; a counterexample for both cases
can be found in \cite{rings}.

A loop is called a \emph{Moufang} loop if every element is Moufang, and
a loop is called an \emph{extra} loop if every element is extra.
A loop is extra if and only if it is Moufang and every square is
in the nucleus \cite{CR}. Thus by the lemma, a Buchsteiner loop is
extra if and only if it is Moufang. On the other hand, we have the
following.

\begin{lem}
\lemlabel{21}
Every extra loop is a Buchsteiner loop.
\end{lem}

\begin{proof}
Since each element $x$ of an extra loop is flexible, that is,
$L_x R_x = R_x L_x$, we have $\Buch(x) = \Ex(x)$.
\end{proof}

Are there any non-extra Buchsteiner loops? We have observed that
in such a loop there have to exist elements without ``nice''
properties. This rules out Moufang loops, but
it does not rule out CC loops. Now, a loop is extra if and only
if it is both Moufang and CC. On the other hand, a CC loop
$Q$ is extra if and only if (A) $x^2\in N(Q)$ for all $x\in Q$,
and (B) every element is IP (or LIP, or RIP, or flexible,
or left or right alternative). We shall show below that a CC loop
$Q$ is a Buchsteiner loop if and only if condition (A) holds.
This gives a vast class of examples of Buchsteiner loops, since
every loop with $|Q:N|=2$ is a CC loop \cite{gr}. These loops can
be derived from groups in a constructive way \cite{d2} and they are
never extra. A non-CC Buchsteiner loop will be constructed in
\S\secref{extend}.

CC loops were defined independently by Soikis \cite{soi} and by Goodaire
and Robinson \cite{gr}. Proofs for the basic properties established
in these papers can be also found elsewhere, see e.g.~\cite{ccm}.
CC loops are those loops satisfying the LCC and RCC laws:
\begin{gather}
\tag{\LCC} x\cdot yz = ((xy)\rdiv x)\cdot xz \\
\tag{\RCC} zy\cdot x = zx\cdot (x\ldiv (yx)) \,.
\end{gather}
These are respectively equivalent to certain triples being
autotopisms:
\begin{gather}
\tag{\LCC$^{\prime}$}
\mathbf{L}(x) = ( R_x\inv L_x , \, L_x , \, L_x ) \\
\tag{\RCC$^{\prime}$}
\mathbf{R}(x) = ( R_x , \, L_x\inv R_x , \, R_x )
\end{gather}

\begin{lem}\lemlabel{23}
A Buchsteiner loop is an LCC loop if and only if it is an RCC loop.
\end{lem}

\begin{proof}
Using (\buchimplong), we have that in Buchsteiner loops,
(\LCC) is equivalent to
$zx\cdot y = z((xy)\rdiv x)\cdot x$. Replacing $y$
with $x\ldiv (yx)$, we get
$zy\cdot (x\ldiv (yx)) = zx \cdot y$, which is (\RCC).
The argument is clearly reversible.
\end{proof}

\begin{prop}
\prplabel{25}
Let $Q$ be a conjugacy closed loop. Then $Q$ is a Buchsteiner loop if and only
if $x^2 \in N(Q)$ for every $x \in Q$.
\end{prop}

\begin{proof}
The CC loop $Q$ is a Buchsteiner loop if and only if
\[
\mathbf{L}(x)\Buch(x)\inv \mathbf{R}(x)
= (I \,, L_x R_x L_x\inv R_x \,, L_x R_x L_x\inv R_x )
\]
is an autotopism for each $x\in Q$. By Proposition \prpref{atp-cent}
and Corollary \corref{equal-nuc}, this holds if and only if
$L_x R_x L_x\inv R_x = R_{x^2}$ where $x^2\in N(Q)$.
\end{proof}

\section{Inverse properties}
\seclabel{IP}

In a loop $Q$, the conditions
\begin{equation}
\eqnlabel{e16}
I^m(xy)I^{m+1}(x) = I^m(y)
\qquad\text{and}\qquad
J^{m+1}(x) J^{m}(yx) = J^m(y)\,,
\end{equation}
for all $x,y\in Q$,
are equivalent. Indeed, if the former one holds, then for $x' = I^{m+1}(x)$ and $y' = I^m(xy)$
we obtain $J^{m+1}(x')J^m(y'x') = J^{m+1}(x')J^m(I^m(y)) = xy = J^m(y')$.

A loop satisfying these conditions is called an $m$-\emph{inverse loop}. These
loops were introduced by Karkli\v n\v s and Karkli\v n \cite{kk} as a generalization
of the weak \cite{osb} and cross \cite{arz} inverse properties, which
correspond to the cases $m=-1$ and $m=0$, respectively.

By reading $I^m(xy)I^{m+1}(x) = I^m(y)$ as $R_{I^{m+1}(x)}I^m L_x = I^m$,
we get another pair of equivalent forms
\begin{equation}
\eqnlabel{e17}
R_{I^{m+1}(x)} = I^mL_x\inv J^m \textand L_{J^{m+1}(x)} = J^m R_x\inv I^m\,,
\end{equation}
for all $x\in Q$.
In particular, we have the following expressions for the left and right division
operations in $m$-inverse loops:
\begin{equation}
\eqnlabel{e18}
x\ldiv y = J^m(I^m(y) I^{m+1}(x))
\qquad\text{and}\qquad
y\rdiv x = I^m(J^{m+1}(x) J^m(y)).
\end{equation}

Let us reformulate some basic properties as they appear in \cite{kk} and in later works.

\begin{lem}
\lemlabel{31}
Let $Q$ be an $m$-inverse loop. Then
\begin{enumerate}
\item\quad $Q$ is also a $(-2m-1)$-inverse loop, and
\item\quad $I^{3m+1}\in \Aut Q$.
\end{enumerate}
\end{lem}

\begin{proof}
We have $J^{-(2m+1)+1} (x) J^{-(2m+1)}(yx) = I^{2m}(x) I^{2m+1}(yx) = I^m
(I^m(yx)\cdot I^{m+1}(y))\cdot I^{m+1}(I^m(yx)) = I^m (I^{m+1}(y)) = J^{-(2m+1)}(y)$,
which yields (i).
Setting $x' = I^{m+1}(x)$ and $y' = I^m(xy)$ in $I^{2m}(x') I^{2m+1}(y'x') = I^{2m+1}(y')$
yields $I^{3m+1}(x) I^{3m+1}(y) = I^{3m+1}(xy)$, and so (ii) holds.
\end{proof}

\begin{lem}
\lemlabel{m-inv-atp}
Let $Q$ be an $m$-inverse loop. If $(\alpha,\beta,\gamma)$ is an autotopism,
then so are
\[
(J^{m+1}\beta I^{m+1}, J^m\gamma I^m, J^m \alpha I^m) \textand
(I^m\gamma J^m, I^{m+1}\alpha J^{m+1}, I^m\beta J^m)\,.
\]
\end{lem}

\begin{proof}
In $\alpha(u) \beta(v) = \gamma(uv)$, take $u = I^m(xy)$, $v = I^{m+1}(x)$,
to obtain $\alpha I^m(xy)\cdot \beta I^{m+1}(x) = \gamma I^m(y)$.
Apply $J^m$ and then multiply on the left by
$J^{m+1} \beta I^{m+1}(x)$ to get
$J^{m+1} \beta I^{m+1}(x) \cdot J^m \gamma I^m(y)
= J^{m+1} \beta I^{m+1}(x)\cdot J^m (\alpha I^m(xy)\cdot \beta I^{m+1}(x))
= J^m \alpha I^m(xy)$. The other case is similar.
\end{proof}

A loop is called an \emph{IP-loop} if it consists solely of IP-elements.
In an IP-loop $(xy)\inv = y\inv x\inv$ since
$(xy)\inv x = (xy)\inv(xy \cdot y\inv) = y\inv$,
and so IP-loops satisfy the identities of the weak inverse property:
\begin{equation*} \tag{WIP}
xI(yx) = I(y) \textand J(xy) x = J(y).
\end{equation*}
However the cross inverse property $(xy)I(x) = y$ holds in an IP-loop
only when the loop is commutative. While general $m$-inverse loops have some
common algebraic properties \cite{kk}, they nevertheless
seem rather to be a topic of a combinatorial nature (e.g., see~\cite{ks}
and the subsequent generalizations to quasigroups \cite{sk1,sk2}).

We shall call a loop a \emph{\WIP{k}}, $k \ge 1$, if it is an
$m$-inverse loop, where $m = ((-2)^k -1)/3$. (Note that for $k=1$
we get WIP.)

\begin{prop}
\prplabel{33}
Let $Q$ be a \WIP{k} loop for some $k\geq 1$. Then
\begin{enumerate}
\item\quad $Q$ is a \WIP{h} loop for every $h \ge k$,
\item\quad $I^{2^k} \in \Aut Q$.
\end{enumerate}
\end{prop}

\begin{proof}
If $3m = (-2)^k-1$, then $|3m+1| = 2^k$ and $3(-2m-1) = -2(3m)-3= (-2)^{k+1}-1$.
The statement thus follows from Lemma \lemref{31}.
\end{proof}

We shall refer to \WIP{2} as the \emph{doubly weak inverse property}
and write WWIP:
\[
\tag{\WWIP}
I(xy)I^2(x) = I(y) \qquad\text{and}\qquad J^2(x)J(yx) = J(y)\,.
\]
In the next section we shall show that every Buchsteiner loop
is a WWIP loop. In \S\secref{WIP}, we will examine
Buchsteiner loops that satisfy WIP.

\section{Doubly weak inverse property}
\seclabel{WWIP}

Throughout this section, let $Q$ be a Buchsteiner loop.
Our main goal is to show that $Q$ has WWIP.

Setting $z = I(xy)$ in (\buchid) and rearranging gives the
first equality of
\begin{equation}
\eqnlabel{workhorse}
I(x)x = y\cdot I(xy)x  \qquad \text{and} \qquad x J(x) = xJ(yx)\cdot y\,,
\end{equation}
and the second equality is verified similarly.

\begin{lem}
\lemlabel{techlemma}
For all $x\in Q$,
\begin{enumerate}
\item\quad $J(x)^2 \cdot x = I(x)$ and $x\cdot I(x)^2 = J(x)$,
\item\quad $I(x)\rdiv x = x \ldiv J(x)$,
\item\quad $J(x)^2 = I(x)^2$,
\item\quad $I(x) x = I(x)^2 x^2$ and $xJ(x) = x^2 J(x)^2$,
\item\quad $I(x)x = J(x) J^2(x)$ and $xJ(x) = I^2(x) I(x)$,
\item\quad $I(x)\cdot xJ(x) = J(x)$ and $I(x)x\cdot J(x) = I(x)$,
\item\quad $J(x)I(x)\cdot x = J(x)$ and $x\cdot J(x)I(x) = I(x)$,
\item\quad $J(x) \rdiv x = x \ldiv I(x)$,
\item\quad $I^2(x) = xJ(x)\cdot x$ and $J^2(x) = x\cdot I(x)x$.
\end{enumerate}
\end{lem}

\begin{proof}
For (i): using Lemma \lemref{E}(i) and \lemref{E}(iv), we compute
\[
J(x)^2\cdot x = R_x R_{J(x)} J(x) = E_x\inv J(x)
= J\inv (x) = I(x)\,.
\]
The other equality follows similarly.

We have $L_x\inv R_{J(x)} L_x = R_x\inv R_{J(x)x} = R_x\inv$,
by \eqnref{buch-trans1}. Thus $L_x\inv R_{J(x)} = R_x\inv L_x\inv$.
Applying both sides to $1\in Q$, we obtain $I(x)\rdiv x =
x\ldiv J(x)$, which is (ii).

We obtain (iii) from (i) and (ii).

For (iv): In \eqnref{workhorse}, set $y = x\ldiv J(x)$ to get
$I(x)x = (x\ldiv J(x))\cdot I(J(x)) x = (x\ldiv J(x))\cdot x^2$.
Since (i) gives
$x\ldiv J(x) = I(x)^2$, we have $I(x)x = I(x)^2 x^2$.
The other equality follows from replacing $x$ with $J(x)$.

For (v): Using (iv) and (iii), $I(x) x = I(x)^2 x^2 = J(x)^2 x^2$.
Now (iii), with $x$ replaced by $J(x)$, is $x^2 = J^2(x)^2$,
and so $I(x) x = J(x)^2 J^2(x)^2$. Now the second equality of
(iv), with $x$ replaced by $J(x)$, is $J(x) J^2(x) = J(x)^2 J^2(x)^2$,
and so $I(x) x = J(x) J^2(x)$. The other equality is
obtained by replacing $x$ with $I(x)$.

For (vi): Using (v) and Lemma \lemref{E}(iii) and \lemref{E}(iv),
we compute
\[
I(x)\cdot x J(x) = I(x)\cdot I^2(x)I(x) =
L_{I(x)} L_{I^2(x)} I(x) = E_{I^2(x)} I(x)
= E_x I(x) = J(x)\,.
\]
The other equality is given by the mirror of this argument.

For (vii): Apply (\buchimp) to (vi).

For (viii): This follows immediately from (vii).

For (ix): Compute
$I^2(x) = E_x\inv (x) = R_x R_{J(x)} (x) = x J(x)\cdot x$,
using Lemma \lemref{E}(iv) and \lemref{E}(i).
\end{proof}

\begin{lem}
\lemlabel{eta}
For all $x\in Q$, $I(x) x = x J(x)$.
\end{lem}

\begin{proof}
Using Lemma \lemref{techlemma}(viii) and (\buchid),
we get
\[
I(x)\ldiv I^2(x) = x \rdiv I(x) = (x\cdot xI(x))\rdiv I(x)
= I(x)\ldiv (I(x)x\cdot x)\,,
\]
and so $I^2(x) = I(x)x\cdot x$. But also
$I^2(x) = xJ(x)\cdot x$ by Lemma \lemref{techlemma}(ix).
Thus $I(x)x\cdot x = xJ(x)\cdot x$,
and the proof is complete after canceling.
\end{proof}

In view of the preceding lemma and many subsequent
calculations, it will be useful to set
\[
\tag{$\eta$}
\eta(x) = x J(x) = I(x) x
\]
for all $x\in Q$.

\begin{lem}
\lemlabel{altaut}
For all $x\in Q$, $L(x,x) \in \Aut(Q)$ and
$R(x,x)\in \Aut(Q)$.
\end{lem}

\begin{proof}
Set $u = (J(x)\rdiv x)\cdot x^2$.
Since $J(x)\rdiv x = J(x) I(x)$ (Lemma \lemref{techlemma}(vii)),
we have $u\cdot J(x)x = J(x) I(x) \cdot x^2$.
Applying (\buchimplong), we get $xu\cdot J(x) = (x\cdot J(x)I(x))\cdot x
= I(x)x = \eta(x) = x J(x)$, using Lemmas \lemref{techlemma}(vii) and
\lemref{eta}. Canceling $J(x)$ and then $x$, we have $u = 1$.
By Corollary \corref{buch-pseudo}(ii), $L(x,x)\in \Aut(Q)$. The other
claim follows from $R(x,x) = L(x,x)\inv$ (Theorem \thmref{26}).
\end{proof}

\begin{lem}
\lemlabel{shift}
For all $x\in Q$,
\[
\begin{array}{rrclcrcl}
\mathrm{(i)} & L_{I^2(x)} &=& L_{\eta(x)} L_x & \text{and} &
R_{\eta(x)} R_x &=& R_{J^2(x)}\,, \\
\mathrm{(ii)} & L_x L_{\eta(x)} &=& L_{J^2(x)} & \text{and} &
R_x R_{\eta(x)} &=& R_{I^2(x)}\,, \\
\mathrm{(iii)} & L_x R_{\eta(x)} &=& R_{\eta(x)} L_x & \text{and} &
R_x L_{\eta(x)} &=& L_{\eta(x)} R_x\,.
\end{array}
\]
\end{lem}

\begin{proof}
For (i): We will prove the first equality; the second
will follow by the mirror of the argument.
By Theorem \thmref{26},
$L_{I^2(x)} = L_{I^2(x)} L(x^2,I(x)) R(I(x),x^2)$,
and it will be useful to compute $L_{I^2(x)} L(x^2,I(x))$
and $R(I(x),x^2)$ separately.

Firstly, we use Lemma \lemref{techlemma}(i) to compute
\[
L_{I^2(x)} L(x^2,I(x))
= L_{I^2(x)} L_{x^2\cdot I(x)}\inv L_{x^2} L_{I(x)}
= L_{I^2(x)} L_{I^2(x)}\inv L_{x^2} L_{I(x)} \\
= L_{x^2} L_{I(x)}\,.
\]
By \eqnref{buch-trans2}, $L_{x^2} = L_x R_x\inv L_x R_x$,
while by Lemma \lemref{E}(iii),
$L_{I(x)} = L_x\inv E_{I(x)} = L_x\inv E_x$. Thus
\begin{equation}
\eqnlabel{shift-tmp1}
L_{I^2(x)} L(x^2,I(x)) = L_x R_x\inv L_x R_x L_x\inv E_x \,.
\end{equation}

Next, using Lemma \lemref{techlemma}(i),
$R(I(x),x^2) = R_{x^2\cdot I(x)}\inv R_{I(x)} R_{x^2}
= R_{I^2(x)}\inv R_{I(x)} R_{x^2}$. By \eqnref{buch-trans1},
$R_{x^2} = R_x L_x\inv R_x L_x$. Thus
$R(I(x),x^2) = R_{I^2(x)}\inv R_{I(x)} R_x L_x\inv R_x L_x
= R_{I^2(x)}\inv E_{I(x)}\inv L_x\inv R_x L_x$. We apply
Lemma \lemref{E}(iii) to conclude
\begin{equation}
\eqnlabel{shift-tmp2}
R(I(x),x^2) = R_{I^2(x)}\inv E_x\inv L_x\inv R_x L_x\,.
\end{equation}

Now we put \eqnref{shift-tmp1} and \eqnref{shift-tmp2}
together to get
\[
L_{I^2(x)} = L_x R_x\inv L_x R_x L_x\inv E_x R_{I^2(x)}\inv E_x\inv L_x\inv R_x L_x \,.
\]
Now $E_x R_{I^2(x)}\inv = R_{E_x I^2(x)}\inv E_x = R_x\inv E_x$, using Lemma \lemref{E}(ii) and \lemref{E}(iv). Thus
$L_{I^2(x)} = L_x R_x\inv (L_x R_x L_x\inv R_x\inv) L_x\inv R_x L_x$.
Applying Lemma \lemref{E}(v) to the parenthesized expression,
$L_{I^2(x)} = L_x R_x\inv E_x L_x\inv R_x L_x$.
Now $E_x L_x\inv = L_{J(x)}$, so
$L_{I^2(x)} = L_x R_x\inv L_{J(x)} R_x L_x$.
Finally, by \eqnref{buch-trans2}, $L_{\eta(x)}
= L_x R_x\inv L_{J(x)} R_x$, and so
$L_{I^2(x)} = L_{\eta(x)} L_x$, as claimed.

For (ii) and (iii): Using Lemma \lemref{techlemma}(ix), we rewrite
the first equality of (i) as $\eta(x)\cdot xy = \eta(x) x \cdot y$.
By Lemma \lemref{cyclic-assoc}, we also have
$x\cdot y \eta(x) = xy \cdot \eta(x)$ and
$x\cdot \eta(x) y = x \eta(x) \cdot y$. The first of these is the
first equality of (iii), while the second can be seen to be the
first equality of (ii) once we have observed that
$x \eta(x) = x \cdot I(x)x = J^2(x)$, by Lemmas \lemref{eta}
and \lemref{techlemma}(ix). The other equalities of (ii)
and (iii) similarly follow from the second equality of (i).
\end{proof}

Recall the inner mapping notation $T_x = R_x\inv L_x$.

\begin{lem}
\lemlabel{eta-aut}
For all $x\in Q$,
\begin{enumerate}
\item\quad $T_{\eta(x)} = R(x,x) E_x\inv$, and
\item\quad $T_{\eta(x)} \in \Aut(Q)$.
\end{enumerate}
\end{lem}

\begin{proof}
We compute
\begin{align*}
T_{\eta(x)} &= R_{\eta(x)}\inv L_{\eta(x)}
= R_{\eta(J^2(x))}\inv L_{\eta(J^2(x))}
= R_x\inv R_{J^2(x)} L_{\eta(J^2(x))} \\
&= R_x\inv L_{\eta(J^2(x))} R_{J^2(x)}
= R_x\inv L_{\eta(x)} R_{J^2(x)}
= R_x\inv L_x\inv L_{J^2(x)} R_{J^2(x)}\,,
\end{align*}
using Lemmas \lemref{eta}, \lemref{shift}(ii), \lemref{shift}(iii),
\lemref{eta} again, and \lemref{shift}(ii) again.
Now $E_x L_x R_x = L_{J^2(x)} R_{J^2(x)} E_x$ using
Lemma \lemref{E}(ii) and \lemref{E}(iv). Thus
\[
T_{\eta(x)} = R_x\inv L_x\inv E_x L_x R_x E_x\inv
= R(x,x) E_x\inv \,,
\]
using Lemma \lemref{E}(vi). This establishes (i).
Part (ii) follows from (i), from $E_x\in \Aut(Q)$
(Lemma \lemref{E}(ii)), and from $R(x,x)\in \Aut(Q)$
(Lemma \lemref{altaut}).
\end{proof}

\begin{lem}
\lemlabel{28}
For $a\in Q$,
\[
L_a R_a\inv \in \Aut Q
\quad \Leftrightarrow \quad
T_a \in \Aut Q
\quad \Leftrightarrow \quad
a \in N.
\]
\end{lem}

\begin{proof}
Indeed, if
$L_a R_a\inv \in \Aut Q$, then
$a((xy)\rdiv a) = a(x\rdiv a) \cdot a(y\rdiv a)$ for all $x,y \in Q$.
Setting $y = a$ yields $ax = a(x\rdiv a)\cdot a$, which is
equivalent to $a\cdot xa = ax\cdot a$.
Thus $L_a R_a = R_a L_a$.

Similarly, $L_a\inv R_a \in \Aut Q$ means
$a\ldiv (xy\cdot a) = (a \ldiv (xa))\cdot (a\ldiv (ya))$ for all $x,y \in Q$.
Again, setting $x = a$ gives $ay\cdot a = a \cdot ya$, that is,
$L_a R_a = R_a L_a$.
Thus $L_a R_a\inv \in \Aut Q \Leftrightarrow L_a\inv R_a \in \Aut Q$.

Finally, $L_a R_a\inv \in \Aut Q$ if and only if
\[
\Buch(a) ( L_a R_a\inv,\, L_a R_a\inv,\, L_a R_a\inv )\inv
= ( L_a R_a L_a\inv ,\, L_a\inv ,\, \id_Q ) \in \Atp Q \,.
\]
This last expression is an autotopism if and only if it
is equal to $( R_a ,\, L_a\inv ,\, \id_Q )$. But this
is equivalent to $a\in N_{\mu} = N$, by Corollary
\corref{equal-nuc}.
\end{proof}

\begin{lem}
\lemlabel{eta-nuc}
For all $x\in Q$, $\eta(x) \in N(Q)$.
\end{lem}

\begin{proof}
By Lemma \lemref{eta-aut}(ii), $T_{\eta(x)}\in \Aut(Q)$.
By Lemma \lemref{28}, $\eta(x)\in N(Q)$.
\end{proof}

We are now ready for the main result of this section.

\begin{thm}
\thmlabel{wwip}
Every Buchsteiner loop has the doubly weak inverse property.
\end{thm}

\begin{proof}
By \eqnref{workhorse} and Lemmas \lemref{eta-nuc}, \lemref{shift}(ii)
(with $I^2(x)$ in place of $x$), and \lemref{eta},
\[
\eta(x) = y\cdot R_x I(xy) = y\cdot R_{\eta(I^2(x))} R_{I^2(x)} I(xy)
= y\cdot (I(xy) I^2(x)\cdot \eta(x))\,.
\]
By Lemma \lemref{eta-nuc}, $\eta(x) = (y\cdot I(xy) I^2(x))\cdot \eta(x)$.
Canceling, we get (\WWIP).
\end{proof}

\section{Calculations in isotopes}
\seclabel{isotopes}

Many important results in loop theory are obtained by considering a given equality
within the principal isotopes of a loop $Q$. In this section we shall follow this
pattern. We first obtain that, in fact, a Buchsteiner loop is isomorphic to all of
its isotopes. At the end we shall be able to verify that the factor
of every Buchsteiner loop by its nucleus is an abelian group.

Recall that for a Buchsteiner loop $Q$, we denote by $Q[b]$ the isotope at
$b\in Q$. The operation in $Q[b]$ is defined as
$b\ldiv (bx \cdot y) = (x\cdot yb)\rdiv b$, cf.~\S\secref{basics}.
Also recall that a \emph{G-loop} is a loop which is
isomorphic to all of its loop isotopes. In fact,
for a loop to be a G-loop, it is sufficient for
it to be isomorphic to all of its left and right
loop isotopes \cite{bel}.

\begin{thm}
\thmlabel{wwip-isom}
Let $Q$ be a Buchsteiner loop. Then for each $x \in Q$, there
exists an isomorphism from $Q$ to $Q[x]$. In particular, every
Buchsteiner loop is a G-loop.
\end{thm}

\begin{proof}
By Theorem \thmref{wwip}, $Q$ has WWIP, and so starting
with the autotopism $\Buch(u)$,
we obtain the autotopism
$\hat{\Buch}(u) = (I L_u R_u\inv J, I^2 L_u J^2, I R_u\inv J)$,
using Lemma \lemref{m-inv-atp}. Now consider the autotopism
$(\alpha_u,\beta_u,\gamma_u) =
\Buch(I(\eta(u)))\inv \hat{\Buch}(u)$.
We have
\[
\alpha_u(1) = L_{I(\eta(u))}\inv I L_u R_u\inv J(1)
= I(\eta(u)) \ldiv I(u J(u)) = 1\,,
\]
using direct computation. Also, since $\eta(u)\in N(Q)$
(Lemma \lemref{eta-nuc}), we use Lemmas \lemref{eta}
and \lemref{techlemma}(ix) (with $x = I(u)$) to compute
\begin{align*}
\beta_u(1) &= R_{I(\eta(u))} I^2 L_u J^2(1) = I^2(u) I(\eta(u)) \\
&= I(\eta(u)\cdot I(u)) = I(\eta(I(u))\cdot I(u)) = I^4(u)\,,
\end{align*}
where in the third equality, we are using the identity
$I(x)a\inv = I(ax)$ for any $a\in N(Q)$.
Now applying Lemma \lemref{16}, we have
$(\alpha_u,\beta_u,\gamma_u) =
(\alpha_u, R_{I^4(u)} \alpha_u, R_{I^4(u)} \alpha_u)$.
Therefore $\alpha_{J^4(x)}$ is the desired isomorphism from
$Q$ to $Q[x]$.
\end{proof}

\begin{prop}
\prplabel{big-ident}
Let $Q$ be a loop. The following are equivalent.
\begin{enumerate}
\item $Q$ is a Buchsteiner loop,
\item For all $x,y,u,v\in Q$,
\[
\tag{\buchbigid}
(xy)\ldiv ((xy\cdot u)v) = (u(v\cdot yx)])\rdiv (yx)\,,
\]
\item For all $x,y,z,u,v\in Q$,
\[
\tag{\buchbigimp}
(xy\cdot u)v = xy\cdot z
\qquad\Rightarrow\qquad
u(v\cdot yx) = z\cdot yx\,.
\]
\item For all $x,y\in Q$,
\[
\BBuch(x,y) = (L_{xy}, R_{yx}\inv, L_{xy} R_{yx}\inv)
\]
is an autotopism.
\end{enumerate}
\end{prop}

\begin{proof}
The equivalence of (ii) and (iii) is clear, as is
the fact that (\buchbigid) implies (\buchid), and
so (ii) implies (i). Also, (iv) is just a rewrite
of (ii). To finish the proof, we shall show that
(i) implies (ii).

Suppose $Q$ is a Buchsteiner loop, and fix
$x\in Q$. The isotope $Q[x]$ is also a Buchsteiner
loop by Theorem \thmref{wwip-isom}. Denote its operation
by $\circ$, and its left
and right translations as $\tilde{L}_y$ and $\tilde{R}_y$,
$y\in Q$, respectively. Then
$\tilde{L}_y (z) = L_x\inv L_{xy}(z)$
and $\tilde{R}_y (z) = R_x\inv R_{yx} (z)$.
Thus
\[
\tilde{L}_y\inv ((y\circ u)\circ v) =
L_{xy}\inv L_x L_x\inv L_{x(y\circ u)}(v) =
(xy)\ldiv ((xy\cdot u)v])\,,
\]
and
\[
\tilde{R}_y\inv (u\circ (v\circ y)) =
R_{yx}\inv R_x R_x\inv R_{(v\circ y)x} (u) =
(u(v\cdot yx)) \rdiv (yx)\,.
\]
Therefore (\buchbigid) holds.
\end{proof}

\begin{thm}
\thmlabel{abelian-factor}
Let $Q$ be a Buchsteiner loop. Then $Q/N$ is an abelian group.
\end{thm}

\begin{proof}
We compute the autotopism
\[
\BBuch(x,y) \Buch(xy)\inv =
(I, R_{yx}\inv R_{xy}, L_{xy} R_{yx}\inv R_{xy} L_{xy}\inv )\,.
\]
By Proposition
\prpref{atp-cent}, $R_{yx}\inv R_{xy} = R_a$ for some
$a\in N_{\rho} = N$. Evaluating at $1\in Q$, we have
$a = (xy)\rdiv (yx) \in N$, that is, $xy \equiv yx$ mod $N$
for all $x,y\in Q$. Thus $Q/N$ is an abelian group.
\end{proof}

\begin{cor}
\corlabel{Alr}
A Buchsteiner loop $Q$ is an $A_{l,r}$-loop.
\end{cor}

\begin{proof}
Since $Q/N$ is an abelian group, we certainly have
$(J(y)\rdiv x)\cdot xy \equiv 1$ mod $N$. But then
$Q$ is an $A_l$-loop by Corollary \corref{buch-pseudo}.
The rest follows from Theorem \thmref{26}.
\end{proof}

We conclude this section with a description of a normal
subloop of a Buchsteiner loop which characterizes the
centers of the left and right multiplication groups.

For a loop $Q$, set
\[
\tag{\dM}
M(Q) = \setof{a\in Q}{L_a \in \RMlt}
= \setof{a\in Q}{T_a \in \RMlt_1}\,.
\]
To see that these define the same set, note that for
$a\in Q$, $L_a = R_a \psi$ for some $\psi\in \RMlt_1$
if and only if $T_a \in \RMlt_1$. In addition, let
\[
\varGamma(Q) = \setof{a\in Q}{a = \phi(1)\text{ for some }\phi\in
\LMlt\cap \RMlt}\,,
\]
that is, $\varGamma(Q)$ is the orbit of $\LMlt\cap \RMlt$
containing the neutral element $1\in Q$. Note that
$M(Q)\subseteq \varGamma(Q)$.

\begin{prop}
\prplabel{M}
Let $Q$ be a Buchsteiner loop. Then
\begin{enumerate}
\item $M(Q) = \varGamma(Q) = \setof{a\in Q}{R_a \in \LMlt}
= \setof{a\in Q}{T_a \in \LMlt_1}$,
\item $M(Q)$ is a normal subloop of $Q$,
\item $M \leq Z(N)$,
$Z(\LMlt) = R_{(M)}$, and $Z(\RMlt)= L_{(M)}$
\end{enumerate}
\end{prop}

\begin{proof}
Part (i) follows from $\LMlt_1 = \RMlt_1$ (Theorem \thmref{26}).
Since $\LMlt \cap \RMlt$ is a normal subgroup of $\Mlt Q$
(Corollary \corref{normal}), its orbits form a block system
of $\Mlt(Q)$, and so (ii) holds.

For $a\in M$, $L_a\inv R_a \in \LMlt_1 \le \Aut Q$,
since $Q$ is an $A_l$-loop (Corollary \corref{Alr}). Thus
$a \in N$ by Lemma \lemref{28}, and so $M \subseteq N$.
Now for $a\in M$, $c\in N$, $ca = R_a c = L_a \varphi(c) = ac$
for some $\varphi\in \LMlt_1$, and so $M\leq Z(N)$.
The remaining assertions of (iii) follow from Lemma
\lemref{centralizers}.
\end{proof}

\section{Weak inverse property}
\seclabel{WIP}

In this section, we shall describe those Buchsteiner loops with
the weak inverse property, and make further remarks about
Buchsteiner CC loops. A parallel development can be found
in \cite{dni}.

\begin{thm}
\thmlabel{wip-buch}
Let $Q$ be a WIP Buchsteiner loop. Then $Q$ is a CC loop.
\end{thm}

\begin{proof}
Starting with the Buchsteiner autotopism $\Buch(x)$, we
obtain from Lemma \lemref{m-inv-atp} (with $m = -1$) that
$( R_x\inv, I L_x R_x\inv J, I L_x J)$ is an autotopism
for each $x$. Now $I L_x J = R_{I^2(x)}\inv$ by WWIP,
and $I R_x\inv J = L_x$ by WIP. Taking inverses, we have
that $( R_x, L_x\inv R_{I^2(x)}, R_{I^2(x)})$ is an autotopism
for each $x\in Q$. Now since $Q/N$ is an abelian group,
$I^2(x) = xn$ for some $n\in N$, and so
$R_{I^2(x)} = R_x R_n$. Then $( R_x, L_x\inv R_x, R_x)$
is an autotopism for each $x\in Q$. But by (RCC'), this
implies $Q$ is an RCC loop. By Lemma \lemref{23}, $Q$ is
a CC loop.
\end{proof}

By Proposition \prpref{25}, a WIP Buchsteiner loop has every
square in its nucleus. However, using, for instance, a finite
model builder like Mace4 \cite{Mace4}, it is easy to find
examples of CC loops of order $16$ with nuclear squares, but
which do not have the WIP. Thus the variety of Buchsteiner CC
loops, which we denote here by \textbf{BuchCC}, is wider than
the variety of WIP Buchsteiner loops, here denoted by
\textbf{BuchWIP}.

In the other direction, let $Q$ be a Buchsteiner loop with
two-sided inverses, that is, in which $J(x) = I(x)$ for all
$x\in Q$. We denote the variety of such loops by
\textbf{Buch2SI}. Then the identity \eqnref{workhorse} becomes
$y \cdot I(xy)x = 1$, which is WIP. By Theorem \thmref{wip-buch},
$Q$ is a CC loop. In CC loops, the condition of having
two-sided inverses is equivalent to power-associativity \cite{kun}.
A detailed structure theory for power-associative CC loops,
including those with the WIP is given in \cite{ccp}.

Narrower still is the variety of Buchsteiner loops with
central squares, denoted by \textbf{BuchCS}. If each $x^2$ is central,
then $x^2 x = x x^2$,
and this identity is equivalent in CC loops to power associativity \cite{kun}.
A power associative CC loop has nuclear squares if and only
if it has the WIP, but there exist power associative
CC loops with WIP which squares which are not central \cite{ccp}.

Summarizing, we have the following proper inclusions among
varieties of Buchsteiner CC loops:
\begin{center}
\begin{tabular}{ccccccc}
\textbf{BuchCS} & $\subset$ & \textbf{Buch2SI} &
$\subset$ & \textbf{BuchWIP} & $\subset$ & \textbf{BuchCC}
\end{tabular}
\end{center}

\section{Associator calculus}
\seclabel{associator}

Let $Q$ be a loop. For $x,y,z \in Q$ define the \emph{associator}
$[x,y,z]$ by
\begin{equation}
\eqnlabel{associator}
(x \cdot yz)\, [x,y,z] = xy \cdot z
\end{equation}
We define the \emph{associator subloop} $A(Q)$ to be the smallest
normal subloop of $Q$ such that $Q/A(Q)$ is a group. Equivalently,
$A(Q)$ is the smallest normal subloop of $Q$ containing all associators.

The following observations, as well as Lemma \lemref{QA-group} below,
are based upon \cite[Lemma 4.2]{ccd}.

\begin{lem}
\lemlabel{assoc-basics}
Let $Q$ be a loop.
\begin{enumerate}
\item $[ax,y,z] = [x,y,z]$ for all $a \in N_{\lambda}$, $x,y,z\in Q$,
\item $[x,y,z] = [x,y,za]$ for all $a \in N_{\rho}$, $x,y,z\in Q$,
\item $[xa,y,z] = [x,ay,z]$ for all $a \in N_\mu \cap N_\lambda$, $x,y,z\in Q$,
\item $[x,ya,z] = [x,y,az]$ for all $a \in N_\mu \cap N_\rho$, $x,y,z\in Q$.
\end{enumerate}
If $N(Q) \unlhd Q$, then $[x,y,z]$ depends only on $xN, yN$, and $zN$.
\end{lem}

\begin{proof}
Parts (i) and (ii) are immediate from the definitions. For (iii),
we have $(xa\cdot yz)[xa,y,z] = (xa\cdot y)z = (x\cdot ay)z
= (x\cdot (ay)z)[x,ay,z]$. Since $xa\cdot yz = x\cdot (ay)z$, we
may cancel to get $[xa,y,z] = [x,ay,z]$, as claimed. The proof
of (iv) is similar.

Now if $N\unlhd Q$, then $xN = Nx$ for all $x \in Q$. Thus for
every $a \in N$ there exist $a_1, a_2 \in N$ such that
$[xa,y,z] = [a_1x,y,z]$ and $[x,y,az] = [x,y,za_2]$.
Applying (i)-(iv), we have the remaining assertion.
\end{proof}

In particular, if $N \unlhd Q$, we may regard $[-,-,-]$ as a mapping
$(Q/N)^3 \to A(Q)$. If $Q/N$ is, in addition, a group, then this
allows us to write, for instance,
$[x\inv,y,z]$ instead of $[I(x),y,z]$ or $[J(x),y,z]$. Similarly,
we can ignore how elements in arguments are associated;
for instance, we may write $[xyz,u,v]$ instead of $[x\cdot yz,u,v]$.

In a loop $Q$ with normal nucleus $N = N(Q)$, for each $x\in Q$,
we denote the restriction of $T_x$ to $N$ by
\[
\tau_x = T_x |_N \,.
\]
The first two parts of the following lemma have been observed many times.

\begin{lem}
\lemlabel{nuc-normal}
Let $Q$ be a loop with $N(Q)\unlhd Q$. Then
\begin{enumerate}
\item for each $x\in Q$, $\tau_x \in \Aut(N(Q))$,
\item $\tau : Q\to \Aut(N(Q)); x\mapsto \tau_x$ is a homomorphism,
\item $A(Q) \leq \ker(\tau)$,
\item $\tau$ descends to a homomorphism
$\bar{\tau} : Q/A(Q) \to \Aut(N(Q))$.
\end{enumerate}
\end{lem}

\begin{proof}
If $a,b \in N$, then for all $x\in Q$, $T_x(a)T_x(b)x = T_x(a)(xb) = (T_x(a)x)b
= x \cdot ab = T_x(ab)x$. This establishes (i).

Next, we have
$T_x T_y(a)\cdot xy = T_x(T_y(a))x\cdot y = x T_y(a) y = xy \cdot a
= T_{xy}(a)\cdot xy$ for $a\in N$, $x,y\in Q$. Thus $T_x T_y(a) =
T_{xy}(a)$, and so also $T_{x\ldiv y}(a) = T_x\inv T_y(a)$ and $T_{x\rdiv y}(a)
= T_x T_y\inv (a)$. This proves (ii).

Now set $s = [x,y,z]$. Then for $a\in N(Q)$, we
compute $(x \cdot yz)\cdot as =
(x\cdot yza)s = xy \cdot za = (xy\cdot z)a =
(x\cdot yz)\cdot sa$. Canceling, we have $as = sa$, that is,
$s\in \ker(\tau)$. Since $A(Q)$ is the smallest normal subloop
containing every associator, we have (iii).

Finally, (iv) follows immediately from (ii) and (iii).
\end{proof}

It will be convenient to introduce exponent notation for the
action of the group $Q/A(Q)$ on $N(Q)$ as follows:
\[
a^x = a^{xA(Q)} = \tau_x\inv (a) = T_x\inv (a) = x\ldiv (ax)\,,
\]
for $x\in Q$, $a\in N$. The notation $a^x$ allows acting
elements of $Q$ to appear as group term. For instance,
we may write $a^{xyz}$ instead of giving the exponent an
explicit association. Similarly, we will write
$a^{x\inv}$ instead of specifying which of $I(x)$ or
$J(x)$ is meant, since both have the same action upon
$a$ modulo $A(Q)$. Also, note that
\begin{equation}
\eqnlabel{e78}
ax = x(x \ldiv (ax))= xa^x
\textand
xa = ((xa)\rdiv x)x = a^{x\inv} x
\end{equation}
for all $x\in Q$, $a\in N\unlhd Q$.

Now we adjoin to these considerations the condition that $Q/N(Q)$
is a group, that is, $A(Q)\leq N(Q)$. In this case, $A(Q)$
coincides with the smallest subloop containing all associators, by
\cite[Lemma 2.5]{ccp}. This gives some insight into our situation,
but we shall not need this result in what follows.

We denote the restriction to $A(Q)$ of the homomorphism $\tau$ by
\[
\tilde{\tau}_x = \tau_x |_{A(Q)} \,.
\]
for each $x\in Q$.

\begin{lem}
\lemlabel{QA-group}
Let $Q$ be a loop with $A(Q)\leq N(Q)\unlhd Q$. Then
\begin{enumerate}
\item $A(Q) \leq Z(N(Q))$,
\item for each $x\in Q$, $\tilde{\tau}_x \in \Aut(A(Q))$,
\item $\tilde{\tau} : Q \to \Aut(A(Q)); x\mapsto \tilde{\tau}_x$
is a homomorphism,
\item $N(Q) \leq \ker(\tilde{\tau})$,
\item $\tilde{\tau}$ descends to a homomorphism
$\bar{\tilde{\tau}} : Q/N(Q) \to \Aut(A(Q))$.
\end{enumerate}
\end{lem}

\begin{proof}
Part (i) follows from Lemma \lemref{nuc-normal}(iii).
Parts (ii) and (iii) follow from parts (i) and (ii) of
that same lemma. Part (iv) follows from (i), and
(v) follows from (iii) and (iv).
\end{proof}

As with the action of $Q/A(Q)$ upon $N(Q)$, it will be helpful
to use exponent notation for the action of $Q/N(Q)$
upon $A(Q)$:
\[
s^x = s^{x N(Q)} = \tilde{\tau}_x\inv (s) = x \ldiv (sx)
\]
for $x\in Q$, $s\in A(Q)$.

Let $Q$ be a loop such that
$A(Q)\leq N(Q) \unlhd Q$, and let $xy\cdot z = s(x \cdot yz)$. Then
$xy\cdot z = sx \cdot yz = xs^x\cdot yz = x \cdot s^x yz = x \cdot ys^{xy} z
= (x \cdot yz)s^{xyz}$, and so $s^{xyz} = [x,y,z]$. We can thus state:

\begin{lem}
\lemlabel{72}
Let $Q$ be a loop with $N \unlhd Q$. For all $x,y,z \in Q$,
\begin{align*}
xy \cdot z &= (x \cdot yz)[x,y,z] = [x,y,z]^{(xyz)\inv}(x\cdot yz), \\
x\cdot yz &= (xy\cdot z)[x,y,z]\inv = ([x,y,z]\inv)^{(xyz)\inv}(xy\cdot z).
\end{align*}
\end{lem}

Now we are ready to turn to Buchsteiner loops.

\begin{prop}
\prplabel{73}
Let $Q$ be a loop with $A(Q) \le N \unlhd Q$.
Then $Q$ is a Buchsteiner loop if and only if
\[
[x,y,z]^x = [y,z,x]\inv \frall x,y,z\in Q.
\]
\end{prop}

\begin{proof}
We have
$x \ldiv (xy\cdot z) = x\ldiv (x \cdot yz)[x,y,z] = yz[x,y,z]$
and
$(y\cdot zx)\rdiv x = ([y,z,x]\inv)^{x\inv z\inv y\inv}(yz) = yz ([y,z,x]\inv)^{x\inv}$,
by \eqnref{e78}.
\end{proof}

We know, by Theorem \thmref{abelian-factor}, that $Q/N$ is an
abelian group if $Q$ satisfies the Buchsteiner law. Thus the
order in which elements of $Q$ act on $A(Q)$ is of no consequence.

Let $Q$ be a loop with $A(Q) \le N(Q) \unlhd Q$. Then
$L(x,y)(z) = z[x,y,z]\inv$ for all $x,y,z \in Q$.
If, in addition, $Q$ is an $A_\ell$-loop, then also
$L(x,y)\inv (uv) = L(x,y)\inv (u) L(x,y)\inv (v)$,
for all $x,y,u,v \in Q$, and that is the same as
$uv[x,y,uv] = u[x,y,u]v[x,y,v] = uv [x,y,u]^v [x,y,v]$.
Thus
\begin{equation}
\eqnlabel{e79}
[x,y,uv] = [x,y,u]^v [x,y,v]\, \text{ if } A(Q)\le N(Q) \unlhd Q,\  Q \text{ an } A_\ell\text{-loop}.
\end{equation}

Every Buchsteiner loop is an $A_\ell$-loop by Corollary \corref{Alr}.
In the rest of this section $Q$ will denote a Buchsteiner loop.

\begin{lem}
\lemlabel{74}
$[z\inv,x,y] = [x,y,z]$ for all $x,y,z \in Q$.
\end{lem}

\begin{proof}
From Proposition \prpref{73}, we have
$[x,y,z\inv ]\inv = [z\inv,x,y]^{z\inv}$. Write this
as $[x,y,z\inv ]^z = [z\inv,x,y]\inv$. Furthermore, $1 = [x,y,z\inv z] =
[x,y,z\inv ]^z [x,y,z] = [z\inv,x,y]\inv [x,y,z]$, using \eqnref{e79}.
\end{proof}

The mapping $(x,y,z)\mapsto (z\inv,x,y)$ defines an action of the
cyclic group $\mathbb{Z}_6$ on $(Q/N)^3$. By the preceding lemma,
associators are invariant under this action. For ease of use, we
record this as follows.

\begin{cor}
\corlabel{75}
For all $x,y,z\in Q$,
\[
[x,y,z] = [z\inv,x,y] = [y\inv,z\inv,x] = [x\inv,y\inv,z\inv ] =  [z,x\inv,y\inv ] = [y,z,x\inv ]\,.
\]
\end{cor}

\begin{lem}
\lemlabel{76}
For all $x,y,z\in Q$,
\[
[x,y,z] = [x,y,z]^{x^2} = [x,y,z]^{y^2} = [x,y,z]^{z^2}\,.
\]
\end{lem}

\begin{proof}
Using Proposition \prpref{73} and Corollary \corref{75},
we get $[x,y,z]^{x^2} = ([y,z,x]^x)\inv = ([x,y\inv,z\inv ]^x)\inv
= [y\inv,z\inv,x] = [x,y,z]$. Further, $[x,y,z]^{y^2} = [y,z,x\inv ]^{y^2} = [y,z,x\inv ]
= [x,y,z]$ and $[x,y,z]^{z^2} = [z,x\inv,y\inv ]^{z^2} = [z,x\inv,y\inv ] = [x,y,z]$.
\end{proof}

\begin{lem}
\lemlabel{77}
For all $x,y,z\in Q$,
\[
[x,y,z]^x = [y,z,x]\inv\,, \quad [x,y,z]^z = [z,x,y]\inv\,,
\quad [x,y,z]^y = [x,y\inv,z]\inv\,,
\]
\[
[x,y\inv,z]^x = [z,x,y]\inv\,,\quad [x,y\inv,z]^y = [x,y,z]\inv\,,
\quad [x,y\inv,z]^z = [y,z,x]\inv\,.
\]
\end{lem}

\begin{proof}
All of these equalities can be easily proved by means of
Proposition \prpref{73} and Corollary \corref{75}.
For example, $[x,y,z]^z = [z,x\inv,y\inv ]^z = [x\inv,y\inv,z]\inv = [z,x,y]\inv$.
\end{proof}

From Lemmas \lemref{76} and \lemref{77} we see that $x$, $y$ and $z$ induce an action
by a group of exponent two on the set $\{[x,y,z]^{\pm 1}, [y,z,x]^{\pm 1},
[z,x,y]^{\pm 1}, [x,y\inv,z]^{\pm 1}\}$.

The next lemma is easy. We need to verify that the behavior described in
\eqnref{e79} is true in every position.

\begin{lem}
\lemlabel{78}
For $x,y,u,v \in Q$,
\[
[uv, x, y] = [u,x,y]^v[v,x,y]\,, \qquad
[x,uv,y] = [x,u,y]^v[x,v,y]\,,
\]
\[
[x,y,uv] = [x,y,u]^v [x,y,v]\,.
\]
\end{lem}

\begin{proof}
From Lemma \lemref{77}, we get
$[x,uv,y]^y = [y,x,uv]\inv = ([y,x,u]^v[y,x,v])\inv =
[x,u,y]^{yv}[x,v,y]^y = ([x,u,y]^v [x,v,y])^y$.
The remaining case can be proved similarly.
\end{proof}

\begin{lem}
\lemlabel{79}
For all $x,y\in Q$, $[x,y,y] = [y,y,x]$.
\end{lem}

\begin{proof}
By Lemma \lemref{77}, $[x,y,y]^y = [y,x,y]\inv = [y,y,x]^y$.
\end{proof}

\begin{prop}
\prplabel{710}
Let $Q$ be a Buchsteiner loop. For all $x,y,z \in Q$,
\[
[y,z,x][x,y,z] = [z,x,y][x,y\inv,z]
\]
\end{prop}

\begin{proof}
We have
\[
[x,x,z]^y [x,y,z]
= [x,xy,z] = [x,yx,z] = [x,y,z]^x [x,x,z] = [x,x,z][y,z,x]\inv\,,
\]
by Lemmas \lemref{77} and \lemref{78}. From that we obtain
\[
[x,x,z]^y = [x,x,z][y,z,x]\inv [x,y,z]\inv\,,
\]
while
\[
[x,x,z]^y = [z,x,x]^y = [x,x,z][z,x,y]\inv [x,y\inv,z]\inv
\]
is a consequence of
\[
[z,x,x]^y[z,x,y] = [z,x,xy] = [z,x,yx] = [z,x,y]^x [z,x,x] =
[x,x,z][x,y\inv,z]\inv\,.
\]
\end{proof}

Express the equality of Proposition \prpref{710} as
\begin{equation}
\eqnlabel{e710}
[x,y\inv,z] = [y,z,x][x,y,z][z,x,y]\inv\,.
\end{equation}
From Corollary \corref{75} we know that
$[x,y\inv,z] = [y,z\inv,x] = [y,x\inv,z]$. The right hand
side of \eqnref{e710} hence retains its value when all arguments
are subjected to a cyclic shift. Thus, for example,
\[
[y,z,x][x,y,z][z,x,y]\inv = [z,x,y][y,z,x][x,y,z]\inv\,,
\]
and therefore
\[
[x,y,z][z,x,y]\inv = [x,y,z][x,y,z]^z = [x,y,z^2] = [x,y,z]\inv [z,x,y]\,.
\]
The leftmost and the rightmost terms of this equation are mutually inverse,
and so the element expressed by this equation has to have exponent $2$.

\begin{prop}
\prplabel{711}
Let $Q$ be a Buchsteiner loop. For all $x,y,z\in Q$,
\[
1 = [x^2,y,z]^2 = [x,y^2,z]^2 = [x,z,y^2]^2 = [x^4,y,z] = [x,y^4,z] = [x,y,z^4]\,.
\]
\end{prop}

\begin{proof}
By the preceding paragraph, $[x,y,z^2]^2 = 1$. It is clear from
Lemma \lemref{77} that the order of an associator element does not change when the
arguments are cyclically shifted. Finally, $[x^4,y,z] = [x^2,y,z]^{x^2}[x^2,y,z]
= [x^2,y,z]^2 = 1$, by Lemmas \lemref{79} and \lemref{76}.
\end{proof}

\begin{thm}
\thmlabel{712}
Let $Q$ be a Buchsteiner loop with nucleus $N(Q)$. Then
$Q/N(Q)$ is an abelian group of exponent $4$.
\end{thm}

\begin{proof}
We have $x\in N(Q)$ if and only if $[x,y,z] = [y,x,z] = [y,z,x]
= 1$ for all $y,z \in Q$. The statement thus follows from Theorem
\thmref{abelian-factor} and Proposition \prpref{711}.
\end{proof}

%

\section{A compatible system of associators}
\seclabel{compatible}

The properties of associators with which we start this section are
included to make the construction of the ensuing example
more transparent.

\begin{lem}
\lemlabel{81}
Let $Q$ be a Buchsteiner loop. For all $x,y,z\in Q$,
\begin{enumerate}
\item
$[x,y,z]^2 = [y,z,x]^2 = [z,x,y]^2$, and $[x^2,y,z] = [y,z,x^2] = [z,x^2,y]$,
\item if $a = [x^2,y,z]$, then $a^x = a^y = a^z = a$,
\item $[x^2,y^2,z] = [x^2,y,z^2] = [x,y^2,z^2] = 1$, and
\item $[x,y,y]^{z^2} = [x,y,y]$ and $[y,x,y]^{z^2} = [y,x,y]$.
\end{enumerate}
\end{lem}

\begin{proof}
From Proposition \prpref{711} we obtain
\[
1 = [x^2,y,z]^2 = ([x,y,z]^x[x,y,z])^2 = [y,z,x]^{-2}[x,y,z]^2\,.
\]
Set now $a = [x^2,y,z]$. Then $a^x = ([x,y,z]^x [x,y,z])^x$ equals $a$
by Lemma \lemref{76}, and $a^y = ([y,z,x]\inv)^y
[x,y,z]^y = [z,x,y]([z,x,y][y,z,x]\inv [x,y,z]\inv)$, by
Proposition \prpref{710} and Corollary \corref{75}. Since we have already
proved the first part of (i), we can replace $[z,x,y]^2$ with
$[x,y,z]^2$, and so $a^y = [x,y,z]^x[x,y,z] = a$. Similarly we get
\begin{align*}
a^z &= ([y,z,x]\inv)^z [x,y,z]^z= [z,x,y] [y,z,x] [x,y,z]\inv [z,x,y]\inv \\
&= [x,y,z][y,z,x]\inv = [x^2,y,z]\,,
\end{align*}
where we have also used the fact
that $a^z$ is of exponent two, by Proposition \prpref{711}.
Furthermore, $[x^2,y,z] = [x^2,y,z]^{x^2}$ by Lemma \lemref{76}, and so
\[
[x^2,y,z] = [y,z,x^2]\inv = [y,z,x^2] = [y,z,x^2]^y =
[z,x^2,y]\inv = [z,x^2,y]\,.
\]
This finishes the proof of (i) and (ii).

Now, $[x^2,y^2,z] = [x^2,y,z]^y[x^2,y,z] = [x^2,y,z]^2 = 1$, which
yields (iii). To get (iv), we compute:
\[
[x,z^2,y][x,y,y] =
[x,z^2,y]^y[x,y,y] = [x,yz^2,y] = [x,z^2,y][x,y,y]^{z^2}\,,
\]
and
\[
[y,x,y][z^2,x,y] = [z^2,x,y]^y[y,x,y] = [yz^2,x,y] =
[y,x,y]^{z^2}[z^2,x,y]\,.
\]
\end{proof}

\begin{lem}
\lemlabel{82}
Let $Q$ be a Buchsteiner loop. Then for all $x,y\in Q$,
\begin{enumerate}
\item $[x,x,y] = [y,x,x]$, $[x,x,y]^x = [x,y,x]\inv$, $[x,y,x]^x = [x,x,y]\inv$,
$[x,x,y]^y = [x,x,y]\inv$ and $[x,y,x]^y = [x,y,x]\inv$,
\item $[x^2,y,x] = [x^2,x,y] = [x,x^2,y] = [x,y,x^2] = [y,x,x^2] = [y,x^2,x]$,
\item $[x,y,x] = [x,x,y][x^2,x,y]$\,
\item $[x,y^2,x] = [x,x,y^2] = [y^2,x,x] = 1 = [x^2,x,y]^2$, and
\item $[x,x,x]^y = [x,x,x] [x,x,y]\inv [x,y,x]\inv$.
\end{enumerate}
\end{lem}

\begin{proof}
We have $[y,x,x] = [x,x,y]$ by Lemma \lemref{79}, and hence we obtain
$[y^2,x,x] = [y,x,x]^y[y,x,x] = [x,x,y]\inv [y,x,x] = 1$.
Now, $[x,x^2,y] = [x,x,y]^x[x,x,y] = [x^2,x,y]$.
That proves (ii) since we can use Lemma \lemref{81}(i).
Next, (i) follows
from Lemma \lemref{77} as $[x,y\inv,x] = [x,y^3,x] = [x,y^2,x]^x[x,y,x] = [x,y,x]$.
We have $[x^2,x,y]^2 = 1$ by Proposition \prpref{711}, and that makes (iii)
and (iv) clear. To prove (v) consider the equalities
\[
[x,x,x][x,y,x]\inv = [x,x,x][x,x,y]^x = [x,x,xy] = [x,x,x]^y[x,x,y]\,.
\]
\end{proof}

From here on, $B$ will be a multiplicative abelian group and $A$ an
additive abelian group, where $B$ acts on $A$ multiplicatively. (Thus
$b(a+a') = ba + ba'$ and $(b'b)a = b'(ba)$ for all $a,a' \in A$ and
$b,b' \in B$.)

We shall construct a mapping $f:B^3 \to A$ such that
\begin{gather}
\eqnlabel{efcond1}
xf(x,y,z) = -f(y,z,x) \text{ for all }x,y,z \in B, \text{ and}\\
\eqnlabel{efcond2}
f(uv,y,z) = vf(u,y,z) + f(v,y,z) \text{ for all } u,v,y,z \in B.
\end{gather}

In this section we shall give a concrete example of such a mapping $f$.
In \S\secref{extend} we shall then extend it to a loop $Q$ in such a way that
$B \cong Q/N$, $A \cong A(Q) = N(Q)$ and that $[x,y,z]$ coincides
with $f(x,y,z)$ (when appropriate identifications are done).
This will give us a Buchsteiner loop, by Proposition \prpref{73}.
If we manage to find $x,y,z \in Q$ with $f(x,y,z) \ne f(y,x,z)$, then
our loop will not be conjugacy closed (we explain this in detail
in \S\secref{extend}).

Set $B = \langle e_1,e_2;$ $e_1^4 = e_2^4 = 1\rangle$ and define
$A$ as a vector space over the two-element field  with basis $c_{ijk}$,
where $k \ge i$ and $i,j,k \in \{1,2\}$. Thus $|A| = 64$.

We shall identify $c_{ijk}$ with $c_{kji}$, which will allow us to
deal with vectors $c_{ijk}$ for all $i,j,k \in \{1,2,3\}$.

To define the action of $B$ on $A$ consider $i,j \in \{1,2\}$, $i \ne j$,
and set
\begin{equation}
\eqnlabel{eact}
\begin{gathered}
e_ic_{iij} = c_{iji},\ \ e_i c_{iji} = c_{iij}, \ \ e_j c_{iij} = c_{iij},
\ \ e_j c_{iji} = e_{iji}, \\
e_i c_{iii} = c_{iii}, \text{ \ and \ } e_jc_{iii} = c_{iii} + c_{iij} + c_{iji}.
\end{gathered}
\end{equation}

The actions of $e_1$ and $e_2$ clearly commute. To see that
we have really obtained an action of $B$ on $A$, it hence
suffices to verify that $e_h^4$
fixes each vector for both $h \in \{1,2\}$. In fact, we already
have $e_h^2 c_{ijk} = c_{ijk}$
for all $i,j,k \in \{1,2\}$. This is almost clear from
the definitions, with perhaps one case requiring our attention:
$e_j^2 c_{iii} = c_{iii} + c_{iij}
+ c_{iji} + c_{iij} + c_{iji} = c_{iii}$.

We have verified that the action is defined correctly.
Furthermore, we can state:

\begin{lem}
\lemlabel{83}
If $b \in B$ and $u \in A$, then $b^2 u = u$.
\end{lem}

Denote by $B_2$ the vector space with basis $e_1$ and $e_2$
over the two-element field. Denote by $\pi$ the projection
$B \to B_2 ; e_i \mapsto e_i$.
By Lemma \lemref{83}, the action of $B$ on $A$ induces an
action of $B_2$ on $A$, .

\begin{lem}
\lemlabel{84}
The elements of $A$ centralized by $B$ form a subspace
generated by $c_{121} + c_{112}$ and $c_{212} + c_{221}$.
\end{lem}

\begin{proof}
Consider $i,j \in \{1,2\}$, $i \ne j$. Clearly $e_i(c_{iji}+c_{iij})
= e_j(c_{iji} + c_{iij}) = c_{iji} + c_{iij}$. Suppose, on the other hand,
that $u \in A$ is centralized by the action. We argue by contradiction,
and assume that in the standard basis the vector $u$ has
different coefficients at $c_{iij}$ and $c_{iji}$. Since $e_ic_{iji} = c_{iij}$,
there must exist a base vector $c \notin \{c_{iji},c_{iij}\}$ such that
$e_ic$ has different coefficients at $c_{iji}$ and $c_{iij}$. However,
no such vector exists.
\end{proof}

Define $C:B_2^3 \to A$ in such a way that $C(a,b,c) = C(c,b,a)$ for all $a,b,c \in B_2$,
$C(a,b,c) = 0$ if one of $a,b,c \in B_2$ is equal to $0$,
$C(e_i,e_j,e_k) = c_{ijk}$ for all $i,j,k \in \{1,2,3\}$, and
the following holds when $i,j\in  \{1,2\}$ and $i\ne j$:
\begin{equation}
\eqnlabel{eC}
\begin{gathered}
C(e_i,e_i,e_i+e_j) = c_{iii}+c_{iji} \textand C(e_i,e_i+e_j,e_i) = c_{iii} + c_{iij},\\
C(e_i,e_i+e_j,e_j) = c_{iij}+c_{jji} \textand C(e_i,e_j,e_i+e_j) = c_{ijj} + c_{iji},\\
C(e_i,e_i+e_j,e_i+e_j) = c_{iii} + c_{iij} +c_{jji} + c_{iji},\\
C(e_i+e_j,e_j,e_i+e_j) = c_{jjj} + c_{iji}, \text{ and}\\
C(e_i+e_j,e_i+e_j,e_i+e_j) = c_{iii} + c_{jjj} + c_{iij} + c_{jji}.
\end{gathered}
\end{equation}

We shall usually assume that $x \in B$ is expressed in the
form $\prod e_i^{\alpha_i}e_i^{2\alpha'_i}$,
$i \in \{1,2\}$, where  $\alpha_i,\alpha_i'\in
\{0,1\}$. Let us have also $y = \prod e_i^{\beta_i}e_i^{2\beta'_i}$
and $z = \prod e_i^{\gamma_i}e_i^{2\gamma'_i}$.

For $h \in \{1,2\}$ define a mapping $s_h: B^3 \to \{0,1\}$ so that
\begin{equation}
\eqnlabel{esdef}
s_h(x,y,z) = \alpha'_h (\beta_2\gamma_1 + \beta_1\gamma_2) +
\beta_h'(\gamma_2\alpha_1 + \gamma_1\alpha_2) +
\gamma_h' (\alpha_2\beta_1 + \alpha_1\beta_2),
\end{equation}
and define $f: B^3 \to A$ by
\begin{multline}
\eqnlabel{efdef}
f(x,y,z) = C(\pi(x),\pi(y),\pi(z)) \\ + s_1(x,y,z)(c_{112} +c_{121})
+s_2(x,y,z)(c_{122} + c_{212}).
\end{multline}

From the definition of $s_h$ we see immediately that
\begin{equation}
\eqnlabel{esper}
s_h(x_1,x_2,x_3) = s_h(x_{\sigma(1)}, x_{\sigma(2)}, x_{\sigma(3)})
\text{ for all permutations } \sigma \in S_3,
\end{equation}
for every $x_1,x_2,x_3 \in Q$ and $i \in \{1,2\}$.

\begin{lem}
\lemlabel{85}
Consider $x,y,z,u \in B$ and $h \in \{1,2\}$. Then
\begin{enumerate}
\item $s_h(xu^2,y,z) = s_h(x,y,z) + s_h(u^2,y,z)$,
\item $s_h(u^2,xy,z) = s_h(u^2,x,z) + s_h(u^2,y,z)$, and
\item $s_h(x^2,y^2,z) = 0$.
\end{enumerate}
\end{lem}

\begin{proof}
Let us have $u^2 = \prod e_i^{2\delta_i}$. Then $s_h(xu_2,y,z) =
(\alpha'h + \delta_h)(\beta_2\gamma_1 + \beta_1\gamma_2) + \beta_h'
(\gamma_2\alpha_1 + \gamma_1\alpha_2) + \gamma'_h(\alpha_2\beta_1 +
\alpha_1\beta_2)$ equals $s_h(u^2,y,z) + s_h(x,y,z)$ since
$s_h(u^2,y,z) = \delta_h(\beta_2\gamma_1 + \beta_1\gamma_2)$. The
rest is clear.
\end{proof}

As an immediate consequence we obtain:

\begin{cor}
\corlabel{86}
For all $x,y,z,u\in B$,
\begin{enumerate}
\item $f(xu^2,y,z) = f(x,y,z) + f(u^2,y,z)$, $f(x,yu^2,z) =
f(x,y,z) + f(x,u^2,z)$ and $f(x,y,zu^2) = f(x,y,z) + f(x,y,u^2)$,
\item the vector $f(u^2,y,z) = f(u^2,z,y)= f(y,u^2,z)= f(z,u^2,y) =
f(y,z,u^2) = f(z,y,u^2)$ is centralized by the action of $B$ on $A$;
\item $f(u^2,xy,z) = f(u^2,x,z) + f(u^2,y,z)$, and
\item $f(x^2,y^2,z) = 0$.
\end{enumerate}
\end{cor}

\begin{prop}
\prplabel{87}
$f(x^2,y,z) = f(x,y,z) + f(y,z,x)$ for all $x,y,z \in B$.
\end{prop}

\begin{proof}
The left hand side of the equality evaluates to $\alpha_1(\beta_2\gamma_1
+ \beta_1\gamma_2)(c_{112} + c_{121}) + \alpha_2(\beta_1\gamma_2 +
\beta_2\gamma_1)(c_{221} + c_{212})$, which is equal to
\begin{equation}
\eqnlabel{eL}
(\beta_1\gamma_2 + \beta_2\gamma_1)(\alpha_1c_{112} + \alpha_1c_{121}
+ \alpha_2c_{221} + \alpha_2c_{212}).
\end{equation}
The right hand side yields the sum
\begin{multline}
\eqnlabel{eR}
C(\alpha_1e_1+\alpha_2e_2,\beta_1e_1+\beta_2e_2,\gamma_1e_1+\gamma_2e_2) + \\
C(\beta_1e_1+\beta_2e_2,\gamma_1e_1+\gamma_2e_2,\alpha_1e_1+\alpha_2e_2),
\end{multline}
as $s_h(x,y,z) + s_h(y,z,x) = 0$ for both $h \in \{1,2\}$.

To compare \eqnref{eL} and \eqnref{eR} we thus need to look only
at $(\alpha_1,\alpha_2)$, $(\beta_1,\beta_2)$ and $(\gamma_1,
\gamma_2)$. The equality is clear if one of these pairs is
equal to $(0,0)$. We can thus assume it is not. Furthermore,
note that then $\beta_1\gamma_2 + \beta_2\gamma_1$ is equal to
zero if and only if $(\beta_1,\beta_2) = (\gamma_1,\gamma_2)$.
In such a case \eqnref{eR} is a sum of two same values, and
as such it is equal to zero as well.
Hence also $(\beta_1,\beta_2) \ne (\gamma_1,\gamma_2)$
can be assumed.

There are three cases to be considered, and these are $\pi(x) =
e_1$, $\pi(x) = e_2$ and $\pi(x) = e_1 + e_2$. Under the assumed
conditions the left hand side \eqnref{eL} is in these cases equal to
$c_{112} + c_{121}$, $c_{221} + c_{212}$ and $c_{112} + c_{121} + c_{221} + c_{212}$,
respectively.

To get the right hand side consider Tables \ref{t1}, \ref{t2} and \ref{t3}.
Each of them tabulates, for the given $\pi(x)$, the values
of both factors that form the sum \eqnref{eR}, running through
the six possible combinations of $(\beta_1,\beta_2)$ and
$(\gamma_1,\gamma_2)$. It is easy to verify that the sum
of the rightmost two columns is always equal to the value given
by the left hand side.
\end{proof}

\begin{table}
\caption{The case $a = \pi(x) = e_1$}
\label{t1}
$\begin{array}{cccc}
b = \beta_1e_1 + \beta_2 e_2 & c = \gamma_1e_1 + \gamma_2e_2 & C(a,b,c) & C(b,c,a)\\
\hline
e_1 & e_2 & c_{112} & c_{121} \\
e_1 & e_1 + e_2 & c_{111} + c_{121} & c_{111} + c_{112} \\
e_2 & e_1 & c_{121} & c_{112} \\
e_2 & e_1+e_2 & c_{122} + c_{121} & c_{122} + c_{112} \\
e_1 + e_2 & e_1 & c_{111} + c_{112} & c_{111} + c_{121} \\
e_1+e_2 & e_2 & c_{112} + c_{122} & c_{122} + c_{121}
\end{array}$
\end{table}

\begin{table}
\caption{The case $a = \pi(x) =  e_2$}
\label{t2}
$\begin{array}{cccc}
b = \beta_1e_1 + \beta_2 e_2 & c = \gamma_1e_1 + \gamma_2e_2 & C(a,b,c) & C(b,c,a)\\
\hline
e_1 & e_2 & c_{212} & c_{122} \\
e_1 & e_1 + e_2 & c_{112} + c_{212} & c_{112} + c_{122} \\
e_2 & e_1 & c_{122} & c_{212} \\
e_2 & e_1+e_2 & c_{222} + c_{212} & c_{222} + c_{122} \\
e_1 + e_2 & e_1 & c_{112} + c_{122} & c_{112} + c_{212} \\
e_1 + e_2, & e_2 & c_{222} + c_{122} & c_{222} + c_{212}
\end{array}$
\end{table}

\begin{table}
\caption{The case $a = \pi(x) = e_1 +e_2$}
\label{t3}
$\begin{array}{cccc}
b  & c & C(a,b,c) & C(b,c,a)\\
\hline
e_1 & e_2 & c_{112} + c_{212} & c_{122} +c_{121} \\
e_1 & e_1 + e_2 & c_{111} + c_{212} & c_{111} + c_{112} +c_{122} + c_{121} \\
e_2 & e_1 & c_{122}+ c_{212} & c_{112}+c_{121} \\
e_2 & e_1+e_2 & c_{222} + c_{121} & c_{222} + c_{122} + c_{112} + c_{212}\\
e_1 + e_2 & e_1 & c_{111} + c_{112} + c_{122} + c_{121} & c_{111} + c_{212} \\
e_1 + e_2 & e_2 & c_{222} + c_{122} + c_{112} + c_{212} & c_{222} + c_{121}
\end{array}$
\end{table}

\begin{lem}
\lemlabel{88}
If $a,b,c \in B_2$, then $aC(a,b,c) = C(b,c,a)$.
\end{lem}

\begin{proof}
We first verify $aC(a,b,b) = C(a,b,b)$. We have
$e_iC(e_i,e_j,e_j) = c_{ijj} = C(e_i,e_j,e_j)$, both
for $i =j$ and $i \ne j$. In the latter case
$e_iC(e_i,e_i+e_j,e_i+e_j) = e_i(c_{iii} + c_{iij}
+ c_{jji} + c_{iji}) = c_{iii} + c_{iji} + c_{jji}
+c_{iij} = C(e_i,e_i+e_j,e_i+e_j)$,
$e_j(e_iC(e_i+e_j,e_i,e_i)) = e_j(c_{iii} + c_{iij})
= c_{iii} + c_{iji} = C(e_i+e_j,e_i,e_i)$ and
$e_j(e_iC(e_i+e_j,e_i+e_j,e_i+e_j)) = e_j
(c_{iii} + c_{jjj} + c_{ijj} + c_{jij} + c_{iji}
+ c_{ijj}) = c_{iii} + c_{iij} + c_{iji} + c_{jjj}
+ c_{ijj} + c_{iji} = C(e_i+e_j,e_i+e_j,e_i+e_j)$.

We can thus assume $b\ne c$, and also $0 \notin \{b,c\}$.
When the penultimate column of Table~\ref{t1} is
multiplied by $a = e_1$, we clearly always get the rightmost
column. Hence the equality holds for $a = e_1$, and
the same approach can be used for the case $a = e_2$.
Let us have $a = e_1+e_2$. We shall use Table~\ref{t3}.
It is enough to observe that $e_1(e_2x) =y$ when
$(x,y)$ takes the values $(C(a,b,c), C(b,c,a))$
in the first four rows, by Lemma \lemref{83}. The second
and the fourth row follow from the formula
$e_i(e_j(c_{iii}+c_{jij})) = e_i(c_{iii} + c_{iij}
+c_{iji} + c_{ijj}) = c_{iii} + c_{iji} + c_{iij}
+c_{ijj}$. Finally, $e_1(e_2(c_{112} + c_{212}))
= e_1(c_{112} + c_{122}) = c_{121} + c_{122}$
and $e_1(e_2(c_{122} + c_{121})) = e_1(c_{212} + c_{121})
= c_{212} + c_{112}$.
\end{proof}

\begin{cor}
\corlabel{89}
$xf(x,y,z) = f(y,z,x)$ for all $x,y,z \in B$.
\end{cor}

\begin{proof}
Let us express $f(x,y,z)$ and $f(y,z,x)$
following the definition \eqnref{efdef}. By using \eqnref{esper} and
Lemma \lemref{84} we
see that the statement holds if and only if
\[
\pi(x)C(\pi(x), \pi(y),\pi(z)) = C(\pi(y),\pi(z),\pi(x))
\]
for all $x,y,z \in B$.
However, that is exactly what is claimed in Lemma \lemref{88}.
\end{proof}

Note that we have just verified \eqnref{efcond1} since $A$ has
exponent $2$. To get \eqnref{efcond2}, several
further verifications that go back to the definition of $f$
are needed if we wish to obtain a direct proof. Such a proof
is possible. But it is not needed, if we can show, without using \eqnref{efcond2},
that there exists a loop $Q$ such that $Q/N$ corresponds to $B$, $A(Q)$
to $A$ and $f$ to $[-,-,-]$, including the actions. Indeed,
in such a case \eqnref{efcond1} guarantees that $Q$ will be a
Buchsteiner loop, by Proposition \prpref{73}, and so
\eqnref{efcond2} will then follow from Lemma \lemref{78}.
We shall see that the suggested path is viable.

\section{Extending the associators to a loop}
\seclabel{extend}

Our goal now is to build a Buchsteiner loop $Q$ with
$Q/N \cong B$ and $A(Q) \cong A$ in such a way that $[-,-,-]$
corresponds to the mapping $f:B^3 \to A$ defined by \eqnref{efdef}.

From \eqnref{e78}, we see that $xa\cdot yb = (xy)(a^yb)$ for all
$x,y \in Q$ and $a,b \in N$ whenever $Q$ is a loop with $N \unlhd Q$.
In our case, we know the action of $Q/N$ on $N$ since we intend
to construct a loop with $N(Q) = A(Q)$. The issue, then, is to
define the product $xy$ for representatives of classes modulo $N$.

If $Q$ is a Buchsteiner loop, then
\begin{equation}
\eqnlabel{e91}
\setof{x^2 a}{x\in Q\text{ and } a\in N(Q)} \unlhd Q
\end{equation}
is a group.
Indeed, it is a normal subloop by Theorem \thmref{712}, and it is a group by
Lemma \lemref{81}(iii).

From considering \eqnref{e91}, we arrive at the idea that
the definitions of products $xy$ can be restricted to
representatives modulo this group.
In other words, we might be able to consider only $x$ and $y$
that correspond to square-free elements of $B$.
The next lemma gives the necessary technical basis for such an
approach.

At this point, we require the notion of loop \emph{commutator},
defined by
\[
xy = yx\cdot [x,y]
\]
for all $x,y$ in a loop $Q$. In a Buchsteiner loop, each
commutator lies in the nucleus, by Theorem
\thmref{abelian-factor}.

\begin{lem}
\lemlabel{91}
Let $Q$ be a Buchsteiner loop. For all $x,y,u,v\in Q$,
\[
xu^2 \cdot yv^2 = (xy\cdot u^2v^2) [u^2,y]^{v^2}[x,y,v^2]([x,y,u^2]\inv
[y,x,u^2])^{v^2}\,.
\]
\end{lem}

\begin{proof}
We start with $xu^2 \cdot yv^2 = (xu^2\cdot y)v^2 [xu^2,y,v^2]\inv$,
and note that $[xu^2,y,v^2] = [x,y,v^2][u^2,y,v^2]^x = [x,y,v^2]$,
by Lemmas \lemref{78} and \lemref{81}.
Furthermore, $xu^2 \cdot y = x \cdot u^2y [x,u^2,y] = x\cdot yu^2
[u^2,y][x,u^2,y] = xy \cdot u^2 [x,y,u^2]\inv [u^2,y][x,u^2,y]$.
\end{proof}

In our case we get a simpler equality
\begin{equation}
\eqnlabel{e92}
xu^2 \cdot yv^2 = (xy \cdot u^2v^2)[u^2,y][x,y,v^2]\,,
\end{equation}
since we assume $A(Q) = N(Q)$, and $A(Q)$ is centralized by
squares, by Lemma \lemref{83}. The other simplification follows
from  Corollary \corref{86}(ii).

The question now is how to compute $[u^2,x]$. Note that
the unique $a \in Q$ satisfying
$x^2\cdot x = (x\cdot x^2)a$ is equal to both $[x,x,x]$ and
$[x^2,x]$. This solves the case $x=a$. To understand the remaining
cases, we start from a general statement.

\begin{lem}
\lemlabel{92}
Let $Q$ be a loop with $A(Q)\le N(Q)$ such that $Q/N$ is an abelian
group. For all $x,y,z\in Q$,
\[
[xy,z] = [x,z]^y[y,z][x,z,y]\inv [x,y,z][z,x,y]\,.
\]
\end{lem}

\begin{proof}
Now, $xy\cdot z = (x\cdot yz)[x,y,z] = x \cdot zy [y,z][x,y,z]
= xz \cdot y [x,z,y]\inv [y,z][x,y,z]$, and $xz\cdot y = zx[x,z]\cdot y
= zx\cdot y [x,z]^y=z\cdot xy [z,x,y][x,z]^y$. These two equalities,
together with Lemma \lemref{QA-group}(i), imply
\[
xy \cdot z= z\cdot xy [x,z]^y[y,z][x,z,y]\inv [x,y,z][z,x,y]\,.
\]
This gives the desired result since $xy\cdot z$ can be also expressed as
$z\cdot xy[xy,z]$.
\end{proof}

\begin{cor}
\corlabel{93}
Let $Q$ be a Buchsteiner loop. For each $x,y,z\in Q$, if
$[x,z,y] = [y,z,x]$, then $[xy,z] = [x,z]^y[y,z][x,y\inv,z]$.
\end{cor}

\begin{proof}
Use Proposition \prpref{710} and Corollary \corref{75}.
\end{proof}

From Lemma \lemref{78} we see that the next statement can be easily proved
by induction on the length of words in $Q/N$. The second part
of the statement follows from Lemma \lemref{79}.

\begin{lem}
\lemlabel{94}
Let $Q$ be a Buchsteiner loop such that $Q/N$ is generated by
the set $\{xN$; $x \in X\}$, where $X\subseteq Q$. If $[x,y,z]
= [z,y,x]$ for all $x,y,z \in X$, then this property holds
for all $x,y,z \in Q$. In particular it is true whenever $Q/N$
can be generated by two elements.
\end{lem}

\begin{lem}
\lemlabel{95}
Let $Q$ be a Buchsteiner loop with elements $x$ and $y$. Then
\begin{enumerate}
\item $[x^2,y] = [x,y]^x[x,y][x,y,x]$,
\item $[x^2,xy] = [x,y]^x[x,y][x,x,x][x,y,x]\inv$,
\item $[x^2y^2,y] = [x,y]^{xy^2}[x,y]^{y^2}[y,y,y][x,y,x]$,
\item $[x^2y^2,x] = [y,x]^y[y,x][x,x,x][y,x,y,]$, and
\item $[x^2y^2,xy] = [x,y]^{xy^2}[y,x]^y
[x,x,x][y,y,y][x,y,x]\inv [y,x,y]\inv$.
\end{enumerate}
\end{lem}

\begin{proof}
To obtain (i), use Lemma \lemref{94} and Corollary \corref{93},
and note that $[x,x\inv,y] =
[x,y\inv,x] = [x,y^2,x][x,y,x] = [x,y,x]$, by Corollary \corref{75} and
Lemma \lemref{82}. For (ii), first note that $[x^2,xy]$ is the inverse
of $[xy,x^2] = [x,x^2]^y[y,x^2][x,y\inv,x^2]$. Now, $[x^2,x]^y
=[x,x,x]^y =[x,x,x][x,x,y]\inv [x,y,x]\inv$, where $[x,y,x]\inv$
gets canceled when $[x^2,y]$ is expressed by (i). Therefore (ii)
follows from Lemma \lemref{82}, since $[x,y,x^2]\inv = [x,x,y][x,y,x]\inv$.

To prove (iii), we first employ Corollary \corref{93} to get $[x^2,y]^{y^2}
[y,y,y][x^2,y^2,y]$. The latter associator is trivial, and the
rest follows from (i) and from Lemma \lemref{82}. To get (iv), we start
from $[x^2,x]^{y^2}[y^2,x]$, which evaluates to $[y,x]^y[y,x][y,x,y][x,x,x]$,
by (ii).

Let us show (v). From Corollary \corref{93} we obtain $[x^2,xy]^{y^2}
[y^2,xy]$, by Lemma \lemref{82}. One part of the expression thus follows
from (ii). Furthermore, $[y^2,xy] = [y,xy]^y[y,xy][y,xy,y]$,
and $[y,xy]$ is an inverse of $[xy,y] = [x,y]^y[y,x,y]$. Thus $[y^2,xy]
= [y,x]^{y^2}[y,x]^ya$, where $a$ is a product of $([y,x,y]\inv)^y[y,x,y]\inv$
and of $[y,xy,y] = [y,y,y][y,x,y]^y$. Hence $a = [y,y,y][y,x,y]\inv$.
\end{proof}

The formulas of Lemma \lemref{95} simplify substantially if $[x,y]=1$
is assumed. We shall do so in our construction. Under simplifying
assumptions, one can express Lemma \lemref{95} by a single formula:

\begin{cor}
\corlabel{96}
Let $Q$ be a Buchsteiner loop. Suppose that $e_1, e_2\in Q$
satisfy $e_1e_2 = e_2e_1$, and that all elements
$c_{ijk}= [e_i,e_j,e_k]$ are of exponent $2$, for all
$i,j,k \in \{1,2\}$. Then
\[
[e_1^{2\alpha_1} e_2^{2\alpha_2}, e_1^{\beta_1} e_2^{\beta_2}]
                 = \prod_{i,j} c_{iji}^{\alpha_i\beta_j}\,,
\]
for all $\alpha_1,\alpha_2,\beta_1,\beta_2 \in \{0,1\}$.
\end{cor}

We are now ready for the construction. We shall define
a loop structure on
$Q = B \times A$ as follows. For all $i\alpha_h,\alpha'_h,
\beta_h,\beta_h' \in \{0,1\}$, where $h \in \{1,2\}$, and for all
$a,b \in A$, we set
\begin{multline}
\eqnlabel{e93}
\left (\prod e_h^{\alpha_h}\prod e_h^{2\alpha_h'},\  a\right )\cdot \left (\prod e_h^{\beta_h}
\prod e_h^{2\beta_h'},\ b\right ) = \\
\left(\prod_h e_h^{\alpha_h +\beta_h}\prod_h e_h^{\alpha_h\beta_h + \alpha_h' + \beta_h'},
\quad D(\alpha_1e_1 + \alpha_2e_2, \beta_1e_1 + \beta_2e_2)\  + \right. \\
\left. \sum_h(\alpha_1\beta_2 + \alpha_2\beta_1)\beta'_h z_h +
\sum_{i,j}\alpha'_i\beta_jc_{iji} +
e_1^{\beta_1} e_2^{\beta_2}a + b \right ) ,
\end{multline}
where $z_1 =c_{112} + c_{121}$, $z_2 = c_{221} + c_{212}$ and $D:B_2 \times B_2
\to A$ is defined in the following way:
\begin{equation}
\begin{gathered}
\eqnlabel{e94}
D(u,v) = 0 \text{ if }u,v \neq e_1 + e_2,\\
D(e_1, e_1+e_2) = c_{112} + c_{121}\, \text { and } D(e_2,e_1+e_2) = c_{122} + c_{212}, \\
D(e_1+e_2,e_1) = c_{112},\  D(e_1+e_2,e_2) = c_{122} \text{ and} \\
D(e_1+e_2,e_1+e_2) = c_{121} + c_{212}.
\end{gathered}
\end{equation}
The arithmetic in the exponents computed modulo $2$. Set
$x = \prod e_h^{\alpha_h+2\alpha'_h}$
and $y = \prod e_h^{\beta_h + 2\beta'_h}$. The formula
\eqnref{e93} can be also expressed as
\begin{equation}
\eqnlabel{e95}
(x,a)\cdot(y,b) = (xy,\, g(x,y)+ya +b), \text{ where } g:B^2 \to A.
\end{equation}

We shall now describe the structure of $g$. Put
$x_0 = (e_1^{\alpha_1}e_2^{\alpha_2},0) = (e_1^{\alpha_1},0)\cdot
(e_2^{\alpha_2},0)$ and $x_1= (e_1^{2\alpha'_1}
e_2^{2\alpha'_2},0) = (e_1^{2\alpha'_1},0)\cdot(e_2^{2\alpha'_2},0)
$, and similarly define $y_0$ and $y_1$.
We have $(x,0)=x_0x_1$ and $(y,0)=y_0 y_1$, by \eqnref{e94}.
Since \eqnref{e92} remains valid when $u^2$ and $v^2$ are
replaced by elements that are squares modulo $N$, we obtain
\begin{equation}
\eqnlabel{e96}
x_0 x_1 \cdot y_0 y_1 = x_0 y_0 \cdot x_1 y_1 [x_1,y_0][x_0,y_0,y_1].
\end{equation}

Now, $[x_0,y_0,y_1]$ should equal
$f(\prod e_i^{\alpha_i},\prod e_i^{\beta_i},
\prod e_i^{2\beta'_i}) = \sum_h \beta'_h(\alpha_1,
\beta_2 + \alpha_2\beta_1)z_h$, since we wish to construct
$Q$ in such a way that $f$ corresponds to $[-,-,-]$.
Furthermore $[x_1,y_0]$ can be enumerated by means of
Corollary \corref{96} as
$\sum \alpha_i'\beta_j c_{iji}$. From $[e_1,e_2] = 1$, one can
derive $[e_1^2,e_2^2]=1$, and so
$x_1y_1 = \prod e_i^{2(\alpha'_i + \beta'_i)}$.

The correction term $D(\pi(x),\pi(y))$ is caused by $x_0 y_0$.
Put, temporarily, $e_i' = (e_i,0)$, $i \in \{1,2\}$. Then, for example,
$e_1' e_2'\cdot e_1' = e_1' \cdot e_2' e_1' [e'_1,e'_2,e'_1]$,
$e_1'\cdot e_2' e_1' = e_1'\cdot e_1' e_2' = e_1'^2 e_2' [e_1',e_1',e_2']$
and $e_1'^2 e_2' = e_2' e_1'^2 [e_1'^2,e_2]$, where the latter
commutator is equal to $[e_1',e_2',e_1']$, by Lemma \lemref{95}. Therefore
$e_1' e_2'\cdot e_1' = e_2 \cdot e_1'^2 [e_1',e_1',e_2']$,
and that is why $D(e_1 + e_2,e_1) = c_{112}$. Other values of $D$
can be computed similarly.

Formula \eqnref{e95} is a standard way to obtain loops from groups.
If $x,y,z \in B$ and $a,b,c \in A$, then
\begin{equation}
\eqnlabel{e97}
[(x,a),(y,b),(z,c)] = (1, g(xy,z) + zg(x,y) - g(x,yz) - g(y,z)].
\end{equation}

Our task hence is to show
\begin{equation}
\eqnlabel{e98}
g(xy,z) + zg(x,y) + g(x,yz) + g(y,z) = f(x,y,z) \frall x,y,z \in B.
\end{equation}

In any of $x$, $y$ or $z$ is equal to $1$, then \eqnref{e98} holds,
and hence we shall assume $1\notin \{x,y,z\}$.

Now, $g(x,y)$ consists of the \emph{correction part}
$D(\sum\alpha_he_h,\sum\beta_he_h)$, the \emph{associator part}
$\sum (\alpha_1\beta_2 + \beta_2\alpha_1)\beta_h' z_h$ and
the \emph{commutator part} $\sum \alpha'_i\beta_jc_{iji}$.
By adding the associator parts of \eqnref{e98} we obtain
\begin{multline}
\eqnlabel{e99}
\sum_h\Big( \big((\alpha_1 + \beta_1)\gamma_2 + (\alpha_2 +
\beta_2)\gamma_1\big)\gamma'_h + \big(\alpha_1\beta_2 + \alpha_2\beta_1
\big) \beta'_h +  \\
\big( \alpha_1(\beta_2 + \gamma_2) + \alpha_2(\beta_1 + \gamma_1)\big)
(\beta'_h + \gamma'_h + \beta_h\gamma_h) + \big(\gamma_1\beta_2 +
                          \beta_1\gamma_2\big)\gamma'_h \Big) z_h,
\end{multline}
and the commutator parts yield
\begin{equation}
\eqnlabel{e900}
\sum_{i,j}\big((\alpha'_i + \beta'_i + \alpha_i\beta_i)\gamma_j
+ \alpha_i'(\beta_j + \gamma_j) + \beta'_i\gamma_j\big)c_{iji} \, + \,
e_1^{\gamma_1}e_2^{\gamma_2} \sum_{i,j}\alpha'_i\beta_i c_{iji}.
\end{equation}

Consider now only the terms in the associator and commutator parts that
involve $\gamma'_h$, $h \in \{1,2\}$. There is none in the commutator
part, and so we get $\sum \gamma'_h(\alpha_1\beta_2 + \alpha_2\beta_1)z_h$.
Terms with $\beta'_h$ occur in the commutator part, but they cancel out.
Hence for $\beta'h$ we obtain $\sum\beta'_h(\alpha_1\gamma_2 + \alpha_2\beta_1) z_h$.
Let us finally consider the terms with $\alpha'_h$. They occur in the
commutator part only and yield.
\begin{equation}
\eqnlabel{e901}
\sum_{i,j} \alpha'_i\beta_j c_{iji} + e_1^{\gamma_1}e_2^{\gamma_2}
\sum_{i,j} \alpha'_i\beta_j c_{iji}.
\end{equation}

We shall treat \eqnref{e901} as a sum of two symmetric factors that
correspond to the choice of $i \in \{1,2\}$. For $i = 1$ we get
\begin{equation}
\eqnlabel{e902}
\alpha_1'\beta_1c_{111} + \alpha'_1\beta_2c_{121} + e_2^{\gamma_2}\alpha_1'\beta_1c_{111}
+ e_1^{\gamma_1}\alpha_1'\beta_2c_{121}.
\end{equation}
It can be easily verified that \eqnref{e902} evaluates to
$\alpha_1'(\beta_2\gamma_1 + \beta_1\gamma_2)z_1$
for all possible choices of $\gamma_1,\gamma_2\in\{0,1\}$.
For example,
for $\gamma_1=\gamma_2 = 1$ we obtain $(\alpha'_1\beta_1 + \alpha'_1\beta_1)c_{111}
+ (\alpha'_1\beta_1 +\alpha'_1\beta_2)c_{112} + (\alpha'_1\beta_2 + \alpha_1'\beta_1)
c_{121}$. By adding \eqnref{e902} to its counterpart where $i=2$,
we thus obtain $\sum_h \alpha_h'(\beta_2\gamma_1 + \beta_1\gamma_2)z_h$,
and similarly for $\beta'_h$ and $\gamma'_h$.

Recall that $f(x,y,z)$ is equal, by the definition \eqnref{efdef},
to the sum of
\begin{equation}
\eqnlabel{e903}
C(\alpha_1e_1 + \alpha_2e_2, \beta_1e_1 + \beta_2e_2, \gamma_1e_1 + \gamma_2e_2)
\end{equation}
and $\sum_h \big ( \alpha_h'(\beta_2\gamma_1 + \beta_1\gamma_2) +
\beta'_h(\alpha_1\gamma_2 + \alpha_2\beta_1) +
\gamma'_h(\alpha_1\beta_2 + \alpha_2\beta_1) \big ) z_h$.

To finish the proof, we hence need to show that
\eqnref{e903} is always equal to
\begin{multline}
\eqnlabel{e904}
D\left(\sum(\alpha_h+\beta_h)e_h, \sum \gamma_he_h\right) + e_1^{\gamma_1}e_2^{\gamma_2}
D\left(\sum\alpha_he_h, \sum\beta_he_h\right) + \\
D\left(\sum\alpha_he_h, \sum(\beta_h + \gamma_h) e_h \right) +
D\left (\sum \beta_he_h,\sum \gamma_he_h\right) + \\
\sum_{i\ne j} \alpha_i(\beta_j + \gamma_j)\sum \beta_h\gamma_hz_h
+ \sum_{i,j} \alpha_i\beta_i\gamma_j c_{iji}.
\end{multline}

Note that \eqnref{e904} consists of six summands, two on
each line. We number them 1, 2, 3, 4, 5, and 6, and we shall use this
numbering in Table~\ref{t5}. This table verifies the equality of
\eqnref{e903} and \eqnref{e904} for every choice of $a,b,c \in B_2$,
where $0\notin \{a,b,c\}$, $a = \sum \alpha_he_h$, $b= \sum \beta_h e_h$
and $c=\sum \gamma_h e_h$. To save space in the table we write $e_3$
in place of $e_1+e_2$. Each row of the table corresponds to one choice
of $(a,b,c)$, and a digit $d$, $1\le d \le 6$, placed in a column
labeled $c_{ijk}$ means that $c_{ijk}$ belongs to the support
of the vector that is obtained by enumeration of the $d$th summand
of \eqnref{e904} for the given choice of $a$, $b$ and $c$. The occurrence
of digit $0$ means that $c_{ijk}$ belongs to the support of
\eqnref{e903}. The equality of \eqnref{e903} and \eqnref{e904} follows
from the fact that in each column and each row one gets an even
number of digits.

\begin{table}
\caption{Verification of the associator formula}
\label{t5}
$\begin{array}
{ccccccccc}
a & b & c &
c_{111}& c_{222} &c_{121}& c_{212}& c_{112}& c_{122}\\
\hline \\
e_1 & e_1 & e_1 &06 &   &     &     &     &     \\
e_1 & e_1 & e_2 &   &   &36   &     &03   &     \\
e_1 & e_1 & e_3 &06 &   &0456 &     &45   &     \\
e_1 & e_2 & e_1 &   &   &03   &     &13   &     \\
e_1 & e_2 & e_2 &   &   &     &     &     &01   \\
e_1 & e_2 & e_3 &   &   &01   &14   &     &04   \\
e_1 & e_3 & e_1 &06 &   &25   &     &0245 &     \\
e_1 & e_3 & e_2 &   &   &26   &     &02   &04   \\
e_1 & e_3 & e_3 &06 &   &0246 &14   &02   &01   \\
e_2 & e_1 & e_1 &   &   &     &     &01   &     \\
e_2 & e_1 & e_2 &   &   &     &03   &     &13   \\
e_2 & e_1 & e_3 &   &   &14   &01   &04   &     \\
e_2 & e_2 & e_1 &   &   &     &36   &0    &3    \\
e_2 & e_2 & e_2 &   &06 &     &     &     &     \\
e_2 & e_2 & e_3 &   &06 &     &0456 &     &45   \\
e_2 & e_3 & e_1 &   &   &     &26   &04   &02   \\
e_2 & e_3 & e_2 &   &06 &     &25   &     &0245 \\
e_2 & e_3 & e_3 &   &06 &14   &0246 &01   &02   \\
e_3 & e_1 & e_1 &06 &   &02   &     &     &     \\
e_3 & e_1 & e_2 &   &   &36   &30   &20   &     \\
e_3 & e_1 & e_3 &06 &   &2456 &01   &45   &13   \\
e_3 & e_2 & e_1 &   &   &03   &36   &     &02   \\
e_3 & e_2 & e_2 &   &06 &     &02   &     &     \\
e_3 & e_2 & e_3 &   &06 &01   &2456 &13   &45   \\
e_3 & e_3 & e_1 &06 &   &01   &26   &0241 &03   \\
e_3 & e_3 & e_2 &   &06 &26   &02   &03   &0242 \\
e_3 & e_3 & e_3 &06 &06 &46   &46   &02   &02
\end{array}
$
\end{table}

We can thus state
\begin{thm}
\thmlabel{97}
Let $B$ be the abelian group $\sbl{e_1,e_2}{e_1^4=e_2^4 = 1}$
and let $A$ be the vector space over $\{0,1\}$ with basis
$c_{111}$, $c_{222}$, $c_{121}$, $c_{212}$, $c_{112}$ and $c_{122}$.
Furthermore, let $f: B^3 \to A$ be the mapping determined by \eqnref{efdef}.
Denote by $Q$ be the loop on $B\times A$ defined by \eqnref{e93}. Then
$Q$ satisfies the Buchsteiner law, $N(Q) \cong A$, $Q/N(Q) \cong B$
and $[(x,a),(y,b),(z,c)] = f(x,y,z)$ for all $x,y,z \in B$
and $a,b,c \in A$.
\end{thm}

Our goal now is to show that the loop $Q$ can be factorized to a loop
of order $64$ such that not all squares are in the nucleus. To find
the appropriate normal subloop we shall start with a general description
of a normal subloop $H \unlhd Q$ that is contained in $N = N(Q)$.

Set $z_1 = c_{121}+c_{112}$, $z_2 = c_{212}+c_{221}$, $s_1 = c_{121}$,
$s_2 = c_{222}$, $c_1 = c_{121}$ and $c_2 = c_{212}$. We see that
$\{z_h, s_h, c_h; $ $ 1 \le h \le 2\}$ is a basis of $A$. For an element
$a = \sum (\vartheta_h z_h + \sigma_h s_h + \gamma_h c_h)$ set
$\vhi_h(a) = \sigma_h + \gamma_h$, $h \in \{1,2\}$. We have thus defined
two linear forms $\vhi_h: A \to \{0,1\}$.

\begin{lem}
\lemlabel{98}
A subspace $H \le A$ is a normal subloop of $Q$ if and only if
$\vhi_h(H) = H$ for both $h \in \{1,2\}$.
\end{lem}

\begin{proof}
From \eqnref{eact} we see that the action of $e_1$ fixes $z_1$, $z_2$, $s_2$
and $c_2$. Furthermore, $e_1$ sends $s_1$ to $s_1+z_1$, and $c_1$ to
$c_1 + z_1$. Hence $e_1\cdot a = a + \vhi(e_1)z_1$ for all $a \in A$, and
the condition $e_1\cdot H = H$ is equivalent to $\vhi_1(H) = H$. The case
$h = 2$ is similar.
\end{proof}

Let $H \le A$ be spanned by $c_{222}$, $c_{122}$, $c_{212}$ and
$c_{111} + c_{121}$,
i.~e.~by $z_2$, $c_2$, $s_2$, $c_2$ and $s_1 + c_1$. Then $H \unlhd Q$,
by Lemma \lemref{98}. By \eqnref{esdef} and \eqnref{efdef},
$[e_1^2,e_2,e_1] = z_1$, $[e_1,e_2,e_1] = c$,
and $[e_1^2e_2,e_1,e_2] = z_1 + c_2$. None of these elements belongs to $H$,
and hence none of $e_1^2H$, $e_2H$ and $e_1^2e_2H$ lies in the nucleus
of $Q/H$. On the other hand, $[e_2^2,a,b]$ is a multiple of $z_2 \in H$,
for all $a,b \in A$, and so $e_2^2 H$ is in the nucleus of $Q/H$.
We see that $Q/H$ is a Buchsteiner loop of order 64 such that the
factor over the nucleus is isomorphic to $C_4 \times C_2$. We can
hence conclude this section with the following.

\begin{thm}
\thmlabel{99}
There exists a Buchsteiner loop $Q$ of order $64$ which contains
an element $x$ such that $x^2 \notin N(Q)$.
\end{thm}

\section{Previous work on Buchsteiner loops}
\seclabel{history}

As discussed in the introduction, H. H. Buchsteiner \cite{buch} was the first to
consider loops satisfying the identity (\buchid) or equivalently,
the implication (\buchimp). He showed that all isotopes of a loop
satisfy (\buchid) if and only if the loop itself satisfies
(\buchbigid). He called those loops satisfying the latter identity
``i-loops'', and the bulk of \cite{buch} is about i-loops.
Buchsteiner did not address nor even state the question of whether or not
the isotopy invariance of (\buchid) holds automatically. He
showed that for each i-loop $Q$, $Q/N$ is an abelian group.
He concluded his paper by posing the problem of whether or not
every i-loop is a G-loop.

A. S. Basarab \cite{bao} answered Buchsteiner's question
affirmatively. He showed that every i-loop satisfies the
property we have called here WWIP, and then used this to
show that i-loops are G-loops.

Our approach uses some of the same ingredients as in these papers,
but in a different way to get stronger results. We started with (\buchid)
itself which, on the surface, at least, is a weaker identity than
(\buchbigid). We showed that a loop $Q$ satisfying (\buchid) has WWIP
(Theorem \thmref{wwip}), then
used WWIP to show that $Q$ is a G-loop (Theorem \thmref{wwip-isom}),
then used that fact to prove (\buchbigid) (Proposition \prpref{big-ident}),
and then finally showed that $Q/N$ is an abelian group
(Theorem \thmref{abelian-factor}).

Our construction of a Buchsteiner loop $Q$ in which $Q/N$ achieves
exponent $4$ is partially motivated by our failure to understand a
claim in \cite{buch} that $Q/N$ always has exponent $2$. Some of
the  associator calculus of \cite{buch} is correct, but there turns
out to be a gap which led Buchsteiner to the identity $[x,y,z] = [y,x,z]$.
It is now known that
if this identity holds in a loop $Q$ for which $Q/N_{\lambda}$
is an abelian group, then the loop is LCC \cite{d2}. So had Buchsteiner's
argument been correct, every Buchsteiner loop would be CC. Taking the above
associator identity as a hypothesis, Buchsteiner's subsequent
argument that $Q/N$ has exponent $2$ is valid;
\textit{cf.} Proposition \prpref{25}.

\section{Conclusions and prospects}
\seclabel{conclusions}

We hope that this paper offers sufficient evidence that the variety
of Buchsteiner loops deserves a place at the loop theory table alongside
other better known varieties such as Moufang and CC loops. In this
section we describe additional results which will be appearing elsewhere.

As mentioned in the introduction, our understanding of Buchsteiner loops
has been greatly helped by recent work in conjugacy closed loops. The
two varieties are tied together more strongly by the following.

\begin{thm}\cite{cccenter}
\thmlabel{cccenter}
Let $Q$ be a Buchsteiner loop. Then $Q/Z(Q)$ is a conjugacy closed loop.
\end{thm}

\begin{thm}\cite{rings}
Let $Q$ be a Buchsteiner loop which is nilpotent of class $2$.
Then $Q$ is a conjugacy closed loop.
\end{thm}

The examples we constructed in this paper served two purposes:
firstly, to show that there exist Buchsteiner loops which are not
conjugacy closed, and secondly, to exhibit examples of Buchsteiner
loops $Q$ such that $Q/N(Q)$ has exponent $4$ but not $2$.

It turns out that there are other constructions of non-CC
Buchsteiner loops such that the factor by the nucleus has
exponent $2$. The paper \cite{rings} gives a construction based
on rings which produces of order $128$. Precise bounds on the
orders for examples of various types are now understood as well.

\begin{thm}\cite{buchass}
\thmlabel{orders}
Let $Q$ be a Buchsteiner loop. If $|Q| < 32$, then $Q$ is a CC loop.
If $|Q| < 64$, then $Q/N(Q)$ has exponent $2$.
\end{thm}

The bounds in this result are sharp. As Theorem \thmref{99} shows,
there  exists a Buchsteiner loop of order $64$ such that the factor by
the nucleus has an element of order $4$. There exist non-CC Buchsteiner
loops of order $32$, and in fact, they are now all classified:

\begin{thm}\cite{buchken}
\thmlabel{classification}
Up to isomorphism, there exist $44$ Buchsteiner loops of order $32$
which are not conjugacy closed.
\end{thm}

By a result going back to Bruck \cite{contrib}, any loop
of nilpotency class $2$ necessary has an abelian inner mapping group.
It has been recently established that the converse is false; an
example of a nilpotent loop of class $3$ and order $128$ with an
abelian inner mapping group was given in \cite{trans}. The example
does not belong to any of the standard loop varieties. Such examples
cannot, for instance, be LCC loops, since LCC loops with abelian
inner mapping groups are always of nilpotency class at most $2$
\cite{lc2}.

Nevertheless, we now know that examples of nilpotent loops of
class $3$ with abelian inner mapping groups exist even
within highly structured varieties:

\begin{thm}\cite{buchass}
There exists a nilpotent Buchsteiner loop $Q$ of nilpotency class
$3$ and order $128$ such that $\Inn Q$ is an abelian group.
\end{thm}

Finally, while we know several general ways to construct
Buchsteiner loops $Q$ for which $Q/N$ has exponent $2$,
we do not know of any such general constructions for the case
when $Q/N$ achieves exponent $4$. It is certainly possible that
the specific construction of \S\secref{compatible} and \S\secref{extend}
might be comprehended in more general terms.
Such a general construction could start from the fact that the
subloop generated by all squares and the nucleus is normal and
is a group; see \eqnref{e91}.  In this connection,
we also pose the problem of characterizing those groups that can
appear in such a context.

\end{document}